\documentclass[12pt]{article}
\usepackage{amsmath,amssymb,a4wide,amsthm,color,graphicx,verbatim,algorithmic,algorithm,caption,enumerate,multirow}
\usepackage{epstopdf}
\usepackage{amsfonts}
\usepackage{refcount}
\usepackage{cleveref}
\usepackage{caption}
\usepackage{subcaption}

\usepackage[normalem]{ulem}
\newtheorem{theorem}{Theorem}[section]

\newtheorem{rem}[theorem]{Remark}

\newenvironment{remark}{\begin{rem} \em }{\em \end{rem}}

\newtheorem{proposition}[theorem]{Proposition}
\newtheorem{ex}[theorem]{Example}
\newenvironment{example}{\begin{ex} \em }{\em \end{ex}} 
\newtheorem{definition}[theorem]{Definition}
\newtheorem{assumption}[theorem]{Assumption}

\newcommand{\norm}  [1]{\ensuremath{\left  \|       #1  \right \|       }}

\newcommand{\abs  } [1]{\ensuremath{\left  |       #1  \right |        }}

\usepackage{color}
\definecolor{orange}{RGB}{250, 54, 0}
\definecolor{purple}{rgb}{0.75, 0.0, 1.0}

\definecolor{col_gk}{rgb}{0.5, 0.0, 0.5}

\definecolor{gray}{rgb}{0.52, 0.52, 0.51}

\newcommand{\cl}  {{\rm cl  \,}}
\newcommand{\bd}  {{\rm bd \,}}

\newcommand{\Int} {{\rm int \,}}
\newcommand{\conv}  {{\rm conv \,}}
\newcommand{\cone}{{\rm cone\,}}

\newcommand{\dom}{{\rm dom\,}}
\newcommand{\ri}{{\rm ri\,}}
\newcommand{\X}{\mathcal{X}}
\newcommand{\Y}{\mathcal{Y}}
\newcommand{\Z}{\mathcal{Z}}
\renewcommand{\P}{\mathcal{P}}
\newcommand{\R}{\mathbb{R}}

\newcommand{\T}{\mathsf{T}}
\newcommand{\N}{\mathbb{N}}

\newcommand{\tr}{{\rm tr}}

\DeclareMathOperator*{\argmin}{arg\,min}

\author{Gabriela Kov\'{a}\v{c}ov\'{a} \thanks{Vienna University of Economics and Business, Institute for Statistics and Mathematics, Vienna A-1020, AUT, gabriela.kovacova@wu.ac.at, ORCID ID: 0000-0003-2088-0597} \and Firdevs Ulus \thanks{Corresponding author. Bilkent University, Department of Industrial Engineering, Ankara, 06800 Turkey, firdevs@bilkent.com.tr, ORCID ID: 0000-0002-0532-9927} }

\title{
	Computing {the} recession cone of a convex upper image via convex projection}

\date{\today}

\begin{document}
\maketitle


\begin{abstract} 

\medskip
It is possible to solve unbounded convex vector optimization problems (CVOPs) in two phases: (1) computing or approximating the recession cone of the upper image and (2) solving the equivalent bounded CVOP where the ordering cone is extended based on the first phase \cite{WURKH22}. In this paper, we consider unbounded CVOPs and propose an alternative solution methodology to compute or approximate the recession cone of the upper image. In particular, we relate the dual of the recession cone with the Lagrange dual of weighted sum scalarization problems whenever the dual problem can be written explicitly. Computing this set requires solving a convex (or polyhedral) projection problem. We show that this methodology can be applied to semidefinite, quadratic, and linear vector optimization problems and provide some numerical examples. 


\medskip


\end{abstract}

\section{Introduction} \label{sect:intro}

Multiobjective optimization, in which there are multiple conflicting objective functions to be minimized simultaneously, has been studied extensively in the literature as 
application areas 
{range} from engineering to natural sciences. Vector optimization is a generalization, {where values of the objective function might not be compared in an element-wise fashion. Stated in technical terms, the order relation on the objective space is determined by an ordering cone, which may differ from the positive orthant. Vector optimization plays an important role in application areas such as financial mathematics \cite{FeiRud2017,HLR13}, economics \cite{RudUlu2021}, and game theory \cite{feinstein2023characterizing}, where the ordering cone of the problem is naturally different than the positive orthant.} 

There are various solution concepts and approaches regarding vector optimization problems in the literature, see for instance \cite{Jahn04, Luc88}. {In contrast to single-objective optimization, there is usually no unique optimal objective value. Instead, one is interested in \emph{Pareto optimal} solutions whose objective values are minimal with respect to the order relation of the problem. The set of objective values of all Pareto optimal points is referred to as the \emph{Pareto frontier}. For technical reasons, we will instead work with the \emph{upper image}, the image of the feasible set of the problem plus the ordering cone. Crucially, the Pareto frontier is contained in the boundary of the upper image. } 

In this paper, we focus on convex vector optimization problems (CVOPs), {where the objective function and the feasible set are appropriately convex. Linear vector optimization problems (LVOPs) 
form an important subclass of CVOPs. An important solution concept for LVOPs (see~\cite{Loehne11}) motivated by set optimization aims to generate the Pareto frontier. Since the upper image of an LVOP is known to be a polyhedron, it can be generated by finitely many extreme points (vertices) and extreme directions. Methods and algorithms exist for solving LVOPs in this sense; for more details see e.g.,~\cite{Ben98,Csirmaz13,HLR13,Loehne11,ShaEhr08}.} 
In the case of a CVOP, {the upper image is a convex set, but not necessarily polyhedral. Hence,} it is commonly {infeasible} 
to compute the exact Pareto frontier. Instead, {there are solution concepts that generate approximations to it.} 

Whether a CVOP is \textit{bounded} or \textit{unbounded} determines what solution concepts and what solution methods are available. In vector optimization, a bounded problem is characterized by its upper image (and its Pareto frontier) being a subset of a shifted ordering cone. In the multiobjective case, this simplifies to each objective being bounded from below on the feasible set. {Note that unbounded problems are encountered for instance in the computations of indifference prices under incomplete preferences \cite{RudUlu2021} or in the implementation of the set-valued Bellman principle \cite{KovacovaRudloff2021}.} 

Bounded problems comprise the less challenging class: the Pareto frontier of a bounded LVOP can be generated by finding its extreme points. For a bounded CVOP, one aims to find finitely many {Pareto optimal} elements that generate both inner and outer polyhedral approximation of the Pareto frontier. There are solution algorithms such as \cite{AraUlusUmer2021, DorLohSchWei2021,EhrShaSch11,LRU14} capable of solving bounded CVOPs in this sense.


Unbounded problems present an additional challenge -- one also has to compute (or approximate) the recession directions of the upper image. One of the solution methods for unbounded LVOPs can be found in \cite{Loehne11}. In the first phase, the recession cone of the upper image {is computed by solving a modified LVOP (the so-called \emph{homogeneous problem}) which is of the same dimension as the original problem but known to be bounded}. In the second phase, the ordering cone is replaced by the found recession cone, which transforms the original unbounded LVOP into an equivalent bounded LVOP. Recently, \cite{WURKH22} proposed a solution concept and a solution algorithm to solve unbounded CVOPs. The solution approach is similar to the linear case and consists of two phases. In the first phase, an outer approximation of the recession cone is algorithmically computed. This outer approximation is then used to transform the problem into a bounded one so that existing algorithms can be applied in the second phase.

In this paper, we propose an alternative way of approximating the recession cone of the upper image{, that is, the first phase given in \cite{WURKH22}. In \cite{WURKH22}, this is done by solving Pascoletti-Serafini scalarizations (see \cite{pascoletti1984scalarizing}) and updating an approximation of the desired recession cone in an iterative manner. In this paper, instead of computing an approximation of the recession cone itself as in \cite{WURKH22}, we consider the dual cone of it. We use a characterization of the dual cone from \cite{Ulus18}} given in terms of the well-known weighted sum scalarizations. We observe that for some classes of CVOPs, it is possible to write the dual of the recession cone explicitly. Then, computing this set reduces to solving a bounded convex projection problem \cite{KovacovaRudloff2021}. For the special case of LVOPs, it is possible to compute the recession cone exactly by solving a bounded polyhedral projection problem \cite{LoeWei15}. Moreover, in this case, it is possible to reduce the dimension of the projection problem by one. 

When the dimension of the objective space is two, the procedure simplifies further and it reduces to solving two convex (or linear, if the corresponding problem is linear) scalar optimization problems. Compared to applying the algorithm from \cite{WURKH22} or solving a two-dimensional {LVOP} as in \cite{Loehne11}, solving two convex or linear programs is simpler and more efficient. 

The structure of the paper is as follows. In Section~\ref{sect:prelim} we provide notation and preliminaries. In Section~\ref{sect:problem} we introduce the convex vector optimization problem and the relevant solution concepts. Section~\ref{sect:recCone} introduces a method for approximating the recession cone of the upper image based on its connection to the set of weights for which the weighted sum scalarization is bounded. In Section~\ref{sect:ex} we discuss particular problem classes for which this method yields representation in the form of a convex projection problem. Section~\ref{sect:comp} provides examples.

\section{Preliminaries} \label{sect:prelim}

Let $q \in \mathbb{N}$ and $\R^q$ be the $q$-dimensional Euclidean space. Throughout the paper, we primarily use the $\ell_2$ norm $\|y\| := \norm{y}_2 = \left( \sum_{i=1}^q  \abs{y_i}^2 \right)^{\frac{1}{2}}$ on $\R^q$. We will shortly remark on results under the {$\ell_p$ norm $\norm{y}_p = \left( \sum_{i=1}^q  \abs{y_i}^p \right)^{\frac{1}{p}}$ for $p \in [1, \infty)$} and the $\ell_\infty$ norm $\norm{y}_\infty = \max_{i\in \{1,\ldots,q\}} \abs {y_i}$. The (closed $\ell_2$) ball centered at point $c \in \R^q$ with radius $r > 0$ is denoted by $B(c,r) :=\{y \in \R^q \mid \norm{y-c} \leq r\}$.

The interior, closure, boundary, and convex hull of a set $A \subseteq \R^q$ are denoted by $\Int A, \cl A, \bd A$, and $\conv A$, respectively. The \textit{(convex) conic hull} of $A$,
$$\cone A := \left\lbrace\sum_{i=1}^n \lambda_ia_i \mid n\in \N, \lambda_1,\ldots,\lambda_n \geq 0, a_i\ldots,a_n \in A\right\rbrace,$$
is the set of all conic combinations of points from $A$.
The \textit{recession cone} of a set $A$ is
$$A_{\infty} = \left\lbrace d \in \R^q  \mid  a + \lambda d \in A \quad \forall a \in A, \lambda \geq 0 \right\rbrace.$$
 For two sets $A,B\subseteq \R^q$, their sum is understood as their \textit{Minkowski sum} 
$$A+B:=\{a+b \in \R^q \mid a \in A, b\in B\},$$
and their distance is measured via the \textit{Hausdorff distance}
$$d^H(A,B) := \max \left\lbrace \sup_{a\in A} \inf_{b\in B} \norm{a-b},\sup_{b\in B} \inf_{a\in A} \norm{a-b} \right\rbrace.$$ 
If a different norm is considered, the Hausdorff distance can be defined analogously. We denote by $A-B$, the set $A + (-1)\cdot B = \{a-b \mid a\in A, b\in B\}$.

A set $A \subseteq \R^q$ is a \textit{polyhedron} if it can be identified through finitely many vertices $v_1, \dots, v_{k_v} \in \R^q, k_v \in \N$ and directions $d_1, \dots, d_{k_d} \in \R^q \setminus \{0\}, k_d \in \N\cup\{0\}$ as
\begin{equation} \label{eq:Vrep}
	A = \conv \{v_1, \dots, v_{k_v}\} + \cone \{d_1, \dots, d_{k_d}\}.
\end{equation}
 A polyhedron can also be represented as a finite intersection of halfspaces.

The \textit{dual cone} of a set $A \subseteq \R^q$ is $A^+ := \{ w \in \R^q \mid \forall a \in A: w^\T a \geq 0\}$. A cone $C \subseteq \R^q$ is \textit{nontrivial} if $\{0\} \subsetneq C \subsetneq \R^q$. It is \textit{pointed} if it does not contain any line through the origin. A cone $C \subseteq \R^q$ generates an order on $\R^q$ given through
$$x \leq_C y \iff y \in \{x\} + C$$
for $x, y \in \R^q$. If $C$ is a nontrivial, pointed, convex ordering cone, then $\leq_C$ is a partial order. A function $f: \R^n \to \R^q$ is \textit{$C$-convex} if for all $x, y \in \R^n$, and all $\lambda \in [0, 1]$ it holds
$$f(\lambda x + (1-\lambda)y ) \leq_C \lambda f(x) + (1-\lambda)f(y).$$

\textit{Convex projection} is a problem of the form
\begin{align*}
\text{compute } Y = \left\lbrace y \in \R^m \mid \exists x \in \R^n: (x, y) \in S \right\rbrace,
\end{align*}
where $S \subseteq \R^n\times \R^m$ is a convex feasible set. If the feasible set $S$ is a polyhedron, then the problem is a \textit{polyhedral projection}. Under solving a projection problem, we understand computing the set $Y$ (if polyhedral) or an approximation of it (otherwise) in the sense of finding a representation as in \eqref{eq:Vrep}. More details on polyhedral projections can be found in \cite{LoeWei15}, and on convex projection in \cite{KovacovaRudloff2021,SZC18}.

\section{Problem description} 
\label{sect:problem}
In this section, we introduce a convex vector optimization problem (CVOP) and its upper image. The main object of interest in this work is the recession cone of the upper image. Its importance can be seen by the role it plays in the boundedness properties of CVOPs and the appropriate solution concepts. 

A convex vector optimization problem is
\begin{align*}
\tag{P} \label{eq:P}  \text{minimize } f(x) \quad \text{ with respect to\ } \leq_C\quad\text{ subject to } h(x) \leq 0,
\end{align*}
where $C \subseteq \mathbb R^q$ is a nontrivial, pointed, convex ordering cone with a non-empty interior and $h:\R^n \to \R^m$ and $f: \R^n \to \R^q$ are continuous functions that are $\R^m_+$- and $C$-convex, respectively. We denote the feasible region by $\X := \{x\in \R^n \mid h(x) \leq 0\}$ and its image by $f(\X):=\{f(x) \mid x \in \X\}$. The {\em{upper image}} of \eqref{eq:P} is defined as 
$$\mathcal{P}:=\cl \left( f(\X)+C \right).$$ 
Here we are particularly interested in the recession cone of the upper image, that is, $\P_\infty$. We encounter it within the boundedness notions for CVOPs recalled below.
\begin{definition}\cite[Definitions 4.2, 4.5]{Ulus18}
Problem \eqref{eq:P} is called \emph{bounded} if there is a point $\hat{p} \in \R^q$ such that $\P \subseteq \{\hat{p}\}+C$; and \emph{unbounded} otherwise. Problem \eqref{eq:P} is called \emph{self-bounded} if $\mathcal{P} \neq \R^q$ and there is a point $\hat{p} \in \R^q$ such that $\P \subseteq \{\hat{p}\}+\P_\infty$. 
\end{definition}
Note that for a bounded problem, it holds $\P_\infty = \cl C$, {see~\cite[Lemma 2.2]{KovacovaRudloff2021}.} A bounded problem is, in particular, always self-bounded. However, an unbounded problem can be self-bounded or not. {An illustration is provided in \Cref{fig0}.} We refer readers to~\cite{Ulus18} for {more} examples and more in-depth discussion. 

\begin{figure}[H]
	\centering
		\includegraphics[width=0.4\linewidth]{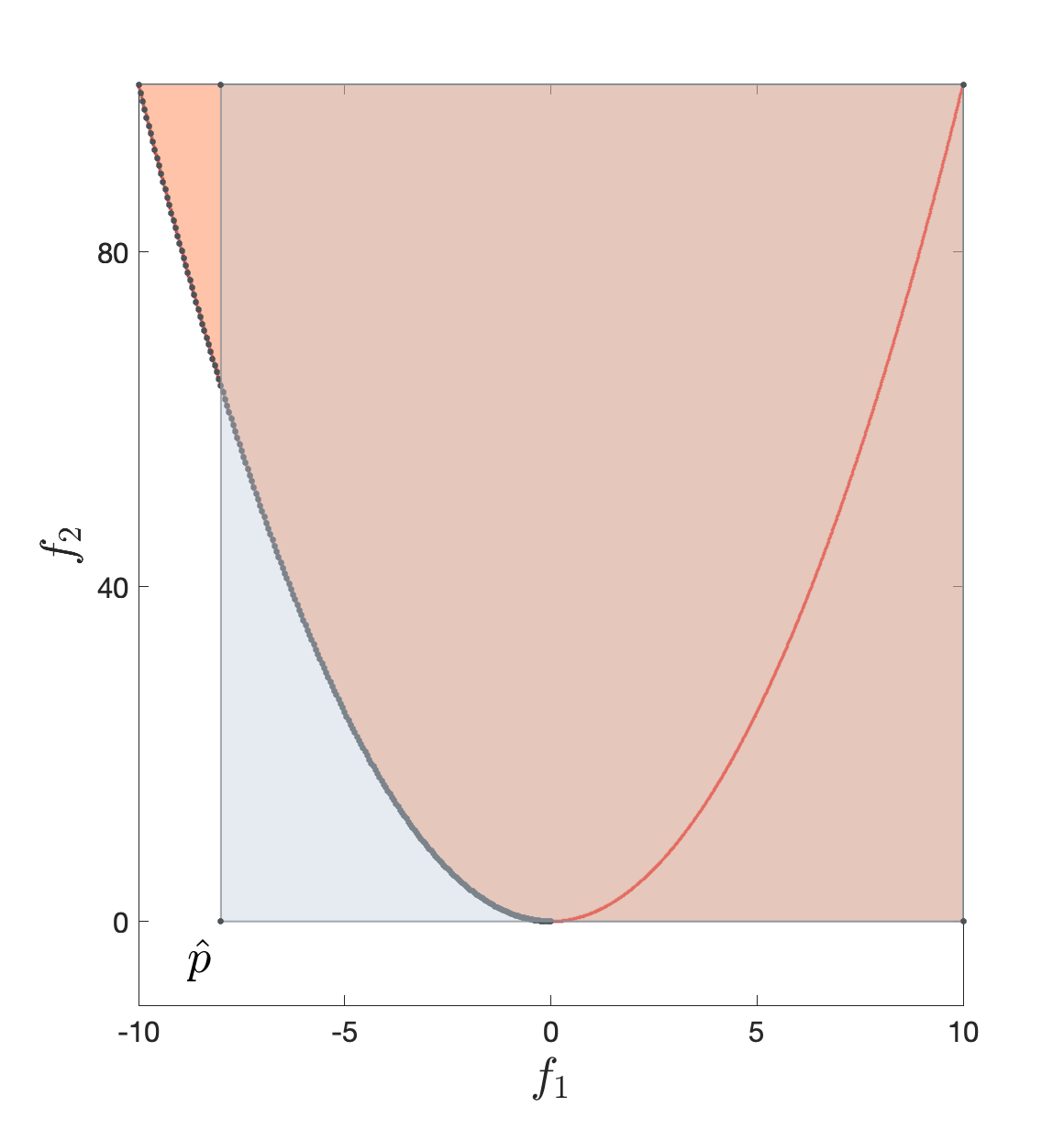}
	
	\caption{{The image $f(\mathcal{X})$ and the upper image $\P$ for the problem with $f(x)= (x,x^2)^\T$, $\mathcal{X} = \R$ and $C = \R^2_+$. The bold line indicates $\bd \P \cap f(\X)$ and the light region shows $\{\hat{p}\}+C$ for $\hat{p} = (-8,0)^\T$. The recession cone of the upper image is $\R^2_+$. This problem is neither bounded nor self-bounded. } }
		\label{fig0}
\end{figure}
 
An appropriate solution concept for a CVOP depends on whether the problem is bounded or not. Solution concepts for bounded CVOPs are proposed in~\cite{AraUlusUmer2021, DorLohSchWei2021,LRU14} and for self-bounded problems in~\cite{Ulus18}. {According to these, a solution consists of finitely many minimal elements on the boundary of the upper image $\P$ which} generate both an inner and an outer approximation of {it}. The self-bounded case, however, contains challenges: In general, it is difficult to check if a CVOP is self-bounded. Moreover, the solution concept of~\cite{Ulus18} includes the recession cone $\P_\infty$. However, computing $\P_\infty$ exactly may not be possible if it is not polyhedral.

Recently, a generalized solution concept was proposed in~\cite{WURKH22}, which includes an approximation of the recession cone $\P_\infty$ of the upper image. This solution concept is tailored for unbounded problems, but it is applicable for a CVOP regardless of whether it is (self-) bounded or not. Similarly to the above, it also yields polyhedral approximations of the upper image. We will provide this solution concept below explicitly as it illustrates the importance of approximating the recession cone $\P_\infty$.

First, we define approximations of a convex cone. Interested readers can also compare this to the definition in~\cite{Doerfler22} for convex sets.
\begin{definition}{\label{defn:approxCone}}
Let $K\subseteq \R^q$ be a convex cone. A finite set $\Y \subseteq \R^q$ is called a \emph{finite $\delta$--outer approximation} of $K$ if $K \subseteq \cone \Y$ and $d^H \left( K \cap B(0,1), \cone \Y \cap B(0,1) \right) \leq \delta$. Similarly, a finite set $\Z \subseteq \R^q$ is called a \emph{finite $\delta$--inner approximation} of $K$ if $K \supseteq \cone \Z$ and $d^H \left( K \cap B(0,1), \cone \Z \cap B(0,1) \right) \leq \delta$.
\end{definition}

Definition~\ref{defn:approxCone} differs slightly from {\cite[Definition 3.3]{WURKH22}}: Here the $\ell_2$ norm (and the corresponding Hausdorff distance) is used, while~\cite{WURKH22} applied the $\ell_1$ norm to measure distance. The $\ell_1$ norm was chosen in~\cite{WURKH22} primarily for algorithmic reasons. Here we opt for the $\ell_2$ norm for pragmatic reasons: Since we will work with dual cones, the $\ell_2$ norm has the advantage of being self-dual. Alternatively, we could work with the pair of $\ell_1$- and $\ell_\infty$ norms, but this would create a cumbersome terminology. When the choice of the norm(s) impacts the results of the paper, we provide corresponding remarks. 

Now we can define a solution of a CVOP, where $c \in \Int C$ with $\Vert c \Vert = 1$ is a fixed element. First, recall that a point $\bar{x} \in \X$ is called a \emph{minimizer} for \eqref{eq:P} if $f(\bar{x})$ is a $C$-minimal element of $f(\X)$, that is, if $(\{f(\bar{x})\}-C\setminus \{0\}) \cap f(\X) = \emptyset$. Similarly, $\bar{x} \in \X$ is called a \emph{weak minimizer} for \eqref{eq:P} if $f(\bar{x})$ is a weakly $C$-minimal element of $f(\X)$, that is, if $(\{f(\bar{x})\}-\Int C) \cap f(\X) = \emptyset$. 

\begin{definition}{\label{defn:solutionconcept}}
	A pair $(\bar{\X},\mathcal{Y})$ is a \emph{(weak) $(\varepsilon,\delta)$--solution} of~\eqref{eq:P} if $\bar{\X} \neq \emptyset$ is a set of (weak) minimizers {for \eqref{eq:P}}, $\mathcal{Y}$ is a $\delta$--outer approximation of $\mathcal{P}_{\infty}$ and it holds
	\begin{align*}
	\mathcal{P} \subseteq \conv {f} (\bar{\X})+\cone \mathcal{Y}-\varepsilon\{c\}.
	\end{align*}
	A (weak) $(\varepsilon,\delta)$--solution $(\bar{\X},\mathcal{Y})$ of~\eqref{eq:P} is a \emph{finite (weak) $(\varepsilon,\delta)$--solution} of~\eqref{eq:P} if the sets $\bar{\X},\mathcal{Y}$ consist of finitely many elements.
\end{definition}

An approach to compute a solution of a CVOP in the sense of \Cref{defn:solutionconcept} is provided in \cite{WURKH22}. It was shown that once an outer approximation of $\P_\infty$ is available, the algorithms for bounded CVOPs can be used to find a solution in the sense of \Cref{defn:solutionconcept}. This is done by transforming the (unbounded) CVOP into a bounded one by replacing the ordering cone with the outer approximation of $\P_\infty$. \cite{WURKH22} also contains an algorithm for computing a finite $\delta$-outer approximation of $\P_\infty$.

In this paper, we provide an alternative approach to compute a polyhedral approximation of $\P_\infty$.
We consider some special classes of CVOPs for which we can compute a finite $\delta$-outer approximation $\Y$ of $\P_\infty$ by solving a particular convex projection problem. For example, we will see that if we consider linear vector optimization problems, then we can compute the exact $\P_\infty$ by solving a polyhedral projection problem. 

\section{Approximating $\P_\infty$ via $\P_\infty^+$} 
\label{sect:recCone}

Let us propose an approach to compute an approximation of $\P_\infty$ by approximating its dual $\P_\infty^+$. It is based on the known close connection between the dual of the recession cone $\P_\infty^+$ and the set of weights for which the weighted sum scalarization of the CVOP is a bounded problem. The boundedness of a scalar (weighted sum scalarization) problem can be verified through the feasibility of its dual problem, assuming strong duality. Expressing the cone $\P_\infty^+$ through a set of weights for which the dual problem is feasible can be interpreted through the lens of a projection problem. This interpretation should become clearer for the particular special cases considered in Section~\ref{sect:ex}. Solving this projection problem provides an inner approximation of $\P_\infty^+$. We show that an inner approximation of $\P_\infty^+$ yields an outer approximation of $\P_\infty$ with an appropriate tolerance.

Let us start by recalling the weighted sum scalarization of \eqref{eq:P}, which is given by
\begin{align*}
\tag{P$_w$} \label{eq:Pw}  \text{minimize } w^\T f(x) \quad\text{ subject to } \quad h(x) \leq 0
\end{align*}
for $w \in \R^q\setminus\{0\}$. It is well known that if $w \in C^+\setminus \{0\},$ then an optimal solution of \eqref{eq:Pw} is a weak minimizer of \eqref{eq:P}, see~\cite[Theorem 5.28]{Jahn04}. On the other hand, for a weak minimizer $\bar{x}\in \R^n$, there exists $w\in C^+$ such that $\bar{x}$ is an optimal solution to \eqref{eq:Pw}, see~\cite[Theorem 5.13]{Jahn04}. This shows us that for CVOPs one is interested in solving \eqref{eq:Pw} for $w \in C^+$. However, the weighted sum scalarization problem may be unbounded for some $w \in C^+$ if \eqref{eq:P} is not bounded. The set of weights for which the weighted sum scalarization is bounded, denoted by 
$$W:= \{w \in C^+ \mid \inf_{x\in \X} w^\T f(x) \in \R \},$$
will play an important role. The following proposition gives a relationship between the dual cone of $\P_\infty$ and $W$.

\begin{proposition} \cite[Proposition 4.12 and Theorem 4.14]{Ulus18} \label{prop:weightset}
	It holds true that $\P_\infty^+ = \cl W$. If \eqref{eq:P} is self-bounded, then $\P_\infty^+ = W$. Furthermore, if $\{0\} \neq \P_\infty^+ = W$, then the problem is self-bounded.
\end{proposition}

Recall that $f$ is a $C$-convex function. Then for all $w\in C^+$ is $w^\T f: \R^n \to \R$ a convex function and hence \eqref{eq:Pw} is a convex optimization problem. 
The Lagrangian $\mathcal{L}: \R^n \times \R^m \to \R$ for \eqref{eq:Pw} is given by 
\begin{equation*}
\mathcal{L}(x,\nu):= w^\T f(x) + \nu^\T h (x)
\end{equation*}
and the dual problem is 
\begin{align*}
\tag{D$_w$} \label{eq:Dw}  {\text{maximize }} g(\nu) \quad\text{ subject to } \nu \in \R^m_+,
\end{align*}
 where the dual objective function $g:\R^m\to \R \cup\{\pm \infty\}$ is defined as $g(\nu):=\inf\limits_{{x \in \R^n}}\mathcal{L}(x,\nu)$. We say that the dual problem is feasible if {the feasible region $\R^m_+ \cap \dom g$ is nonempty, that is, $\{\nu \in\R^m_+ \mid  g(\nu) > -\infty \} \neq \emptyset$.} 
We know that the weak duality holds between the primal and dual problems \eqref{eq:Pw} and \eqref{eq:Dw}, that is, 
 $$p^w := \inf_{x\in \X} w^\T f(x) \geq \sup_{\nu\in\R^m_+} g(\nu) =:d^w.$$
Moreover, we say that the strong duality holds if the value of primal and dual problems are the same, that is, $p^w = d^w$. From now on, we assume the following.

\begin{assumption} \label{assm:const_qual}
 The problem \eqref{eq:P} is feasible and it satisfies a constraint qualification such that the strong duality holds for the pair of \eqref{eq:Pw} and \eqref{eq:Dw} for any $w \in C^+$. 
 \end{assumption}

This assumption is satisfied, for example, if the problem has only affine constraints, or if the (generalized) Slater's condition holds, that is, there exists $\bar{x} \in \ri \X $ such that $h(\bar{x}) < 0$. Strong duality gives us the following result.

\begin{theorem} \label{thm:1}
Suppose \Cref{assm:const_qual} holds. It holds true that
	\begin{equation*}
		W = \{w \in C^+ \mid \text{\eqref{eq:Dw} is feasible}\}.
	\end{equation*}
\end{theorem}

\begin{proof}
Since \eqref{eq:P} is feasible, \eqref{eq:Pw} for any $w\in C^+$ is also a feasible problem. Then, $p^w < \infty$ holds and the weak duality implies that the dual problem \eqref{eq:Dw} is not unbounded. On the other hand, the strong duality implies that \eqref{eq:Pw} is bounded if and only if \eqref{eq:Dw} is feasible. 
\end{proof}

For some classes of convex optimization problems, it is possible to write the constraints of the dual problem \eqref{eq:Dw} explicitly. In the following section, we will consider these classes, for which we will express $\P_\infty^+$ explicitly. This will provide a way to compute $\P_\infty^+$ or an approximation of it. 

Recall that the initial aim was to compute $\P_\infty$ or its outer approximation. If $\P_\infty^+$ is determined by finitely many generators, it is easy to compute (the finitely many generators of) $\P_\infty$. What if its (inner) approximation is available instead? Will a dual cone of an approximation of $\P_\infty^+$ be an approximation of $\P_\infty$? The following proposition provides an answer.

\begin{proposition}\label{prop:aprox_error}
	Let $\Z \subseteq \R^q$ be a finite $\delta$-inner approximation of $\P_\infty^+$ and $\Y$ be a finite set of generating vectors of $(\cone \Z)^+$, that is, $(\cone \Z)^+ = \cone \Y$. Then, $\Y$ is a finite $\delta$--outer approximation of $\P_\infty$.
\end{proposition}

\begin{proof} 
Since $\cone \Z \subseteq \P_\infty^+$, we have $\cone \Y = (\cone \Z)^+ \supseteq \P_\infty$. Let $y \in \cone \Y \cap B(0,1)$ and assume for contradiction that $(\{y\} + B(0,\delta)) \cap (\P_\infty \cap B(0,1)) = \emptyset$. 

First, we show that $(\{y\} + B(0,\delta)) \cap (\P_\infty \cap B(0,1)) = \emptyset$ implies $(\{y\} + B(0,\delta)) \cap \P_\infty = \emptyset$. Assume, also by contradiction, that this is not the case, therefore there exists $b \in B(0,\delta)$ such that $y + b \in \P_\infty \setminus B(0,1)$. Since $\P_\infty$ is a cone, it holds $\lambda (y + b) \in \P_\infty$ for all $\lambda \geq 0$. Consider the convex quadratic optimization problem \begin{equation}\label{eq:quad_pr} 
	\min_{\lambda \geq 0} \Vert \lambda (y + b) - y \Vert^2 = \min_{\lambda \geq 0} \{\lambda^2 \Vert y + b \Vert^2 - 2\lambda y^\T (y+b) + \Vert y \Vert^2\}.
	\end{equation} 
{
The quadratic problem~\eqref{eq:quad_pr} is solved by $\lambda^* = \frac{y^\T (y+b)}{\Vert y + b \Vert^2}$ if $y^\T(y+b) \geq 0$ and by $\lambda = 0$ otherwise. First, consider the case of $\lambda^* < 0$, where we have $\norm{y} \leq \norm{b} \leq \delta$ by monotonicity of the quadratic objective function for $\lambda \in [\lambda^*,  \infty)$. Therefore, vector $0 \in (\{y\} + B(0,\delta)) \cap (\P_\infty \cap B(0,1))$ provides the desired contradiction for the case of $\lambda^*<0$. Second, consider the case $\lambda^* \geq 0$, where the Cauchy-Schwarz inequality yields $\lambda^* (y+b) \in B(0,1)$ as}	

\begin{align*}
\lambda^* = \frac{y^\T (y+b)}{\Vert y + b \Vert^2} \leq \frac{\Vert y \Vert \Vert y + b \Vert}{\Vert y + b \Vert^2} \leq \frac{1}{\Vert y + b \Vert}.
\end{align*} 
Since $\lambda^*$ is an optimal solution of \eqref{eq:quad_pr}, we also obtain $\Vert \lambda^* (y + b) - y \Vert \leq \Vert  (y + b) - y \Vert = \Vert b \Vert \leq \delta$. Therefore, vector $\lambda^* (y+b)  \in (\{y\} + B(0,\delta)) \cap (\P_\infty \cap B(0,1))$ provides the desired contradiction for the case of $\lambda^* \geq 0$, and the implication is proven.

Second, we use $(\{y\} + B(0,\delta)) \cap \P_\infty = \emptyset$ to show that the initial assumption cannot hold. By separation arguments, there exists $w \in \R^q \setminus \{0\}$ such that $w^\T(y - b)< w^\T p$ for all $b \in B(0,\delta), p\in \P_\infty$. In particular, $w \in \P_\infty^+$ and $w^\T(y - b)< 0$ for all $b \in B(0,\delta)$. Without loss of generality, we may assume $\norm{w} = 1$. The choice of $\bar{b} = -\delta w$ shows that it holds $w^\T y < w^\T \bar{b} = -\delta$.
 On the other hand, since $w \in \P_\infty^+ \cap B(0,1)$, there exists $z \in \cone \Z \cap B(0,1)$ such that $\norm{w-z} \leq \delta.$ Since $z \in \cone \Z, y\in \cone \Y$, we have $y^\T z  \geq 0. $ Then, using the Cauchy-Schwarz inequality, we obtain 
\begin{align*}
0 \leq y^\T z = y^\T(z-w) + y^\T w \leq \norm{y} \norm{z-w} + y^\T w < 0, 
\end{align*}
which is a contradiction.
\end{proof}

{Let us now address the issue of the norm used. The above proposition holds for the (self-dual) $\ell_2$ norm. Do we get a similar result for other (dual pairs of) norms? For computational purposes, the pair of $\ell_1$ and $\ell_\infty$ norms with polyhedral unit balls are particularly important. The following remark shows that, in the general case, the tolerance is increased, but by less than a factor of two.}

\begin{remark}\label{rem_1_infty}
{Let $p, r \in [1, \infty]$ satisfy $\frac{1}{p} + \frac{1}{r} = 1$ and consider the dual pair of $\ell_p$ and $\ell_{r}$ norms alongside appropriately adapted Definition~\ref{defn:approxCone} of approximation of a cone. The following can be shown: \textit{If $\Z \subseteq \R^q$ is a finite $\delta$-inner approximation of $\P_\infty^+$ in $\ell_p$ and $\Y$ is a finite set of generating vectors of $(\cone \Z)^+$, then $\Y$ is a finite $\frac{2\delta}{1+\delta}$--outer approximation of $\P_\infty$ in $\ell_{r}$.}

We sketch the proof of this claim. Let $B_p$ and $B_{r}$ denote the closed balls with respect to the $\ell_p$ and $\ell_{r}$ norms. Let $y \in \cone \Y \cap B_{r}(0,1)$. First, we prove by contradiction that $(\{y\} + B_{r}(0, \frac{2\delta}{1+\delta}))  \cap (\P_\infty \cap B_{r}(0,1)) = \emptyset$ implies $(\{y\} + B_{r}(0,\delta)) \cap \P_\infty = \emptyset$: Assume that $y + b \in \P_\infty \setminus B_{r}(0,1)$ for some $b \in B_{r}(0,\delta)$ and consider the convex optimization problem 
\begin{equation}\label{eq:quad_pr_1} 
	\min_{\lambda \geq 0} \norm{ \lambda (y + b) - y }_{r} .
\end{equation} 
Since a coercive function attains a minimum over a closed set, there exists an optimal solution $\lambda^*$ satisfying $\norm{\lambda^* (y+b) - y}_{r} \leq  \norm{b}_{r} \leq \delta$. Hence, $\norm{\lambda^* (y+b) }_{r}  \leq \norm{y}_{r} + \norm{\lambda^* (y+b) - y}_{r}  \leq 1+ \delta$.  The point $\frac{1}{1+\delta}\lambda^* (y+b)  \in \P_\infty \cap B_{r}(0,1)$ provides the desired contradiction since $\frac{1}{1+\delta}\lambda^* (y+b)  \in \{y\} + B_{r}(0, \frac{2\delta}{1+\delta})$ follows from
\begin{align*}
\norm{\frac{1}{1+\delta}\lambda^* (y+b) - y}_{r} \leq \frac{1}{1+\delta}  \norm{\lambda^* (y+b) - y}_{r} + \frac{\delta}{1+\delta}  \norm{y}_{r} \leq \frac{\delta}{1+\delta} + \frac{\delta}{1+\delta}.
\end{align*}

Second, we show that $(\{y\} + B_{{r}}(0,\delta)) \cap \P_\infty = \emptyset$ leads to a contradiction: By a separation argument, there exists $w \in \P_\infty^+ \setminus \{0\}$ such that
$w^\T(y+b) < 0$ for all $b \in B_{{r}}(0, \delta)$ and, therefore, $w^\T y < -\delta$. Since we can without loss of generality assume $w \in P_\infty^+ \cap B_p (0, 1)$, there must exist $z \in \cone \Z \cap B_p (0,1)$ such that $\norm{w - z}_p \leq \delta$. Since $(\cone \Z)^+ = \cone \Y$, we get a contradiction
\begin{align*}
0 \leq z^\T y \leq y^\T (z-w) + y^\T w \leq \norm{y}_{r} \norm{z-w}_p  + y^\T w < 1 \cdot \delta - \delta = 0.
\end{align*}
}

{\begin{remark}
In \cite{walkup1967continuity}, a slightly different Hausdorff distance for closed convex cones $K_1,K_2 \subseteq \R^q$ is defined as $$d^W(K_1,K_2):=\max\{\sup_{k_1 \in K_1 \cap B_p(0,1)} \inf_{k_2\in K_2} \norm{k_1-k_2}_p, \sup_{k_2 \in K_2 \cap B_p(0,1)} \inf_{k_1\in K_1} \norm{k_1-k_2}_p\},$$ where $p \in [1, \infty]$. If this Hausdorff distance is used to define $\delta$-inner and -outer approximations of cones in \Cref{defn:approxCone}, then \Cref{prop:aprox_error} holds, that is, the approximation tolerance is preserved, for any dual pairs of norms by \cite[Theorem 1]{walkup1967continuity}. However, this result cannot be applied to our case since the two distance measures do not coincide in general. To see this, consider $K_1 = \cone \{(0.2, 0.8)^\T\}$ and $K_2 = \cone\{(0.4, 0.6)^\T\} \subseteq \R^2$. If we use the $\ell_1$ norm in both measures, we obtain $d^W (K_1,K_2) = \frac13 $, while $d^H(K_1\cap B_1(0,1),K_2\cap B_1(0,1)) = 0.4$. 
	\end{remark}}
\end{remark}

\Cref{prop:aprox_error} suggests that by computing a $\delta$-inner approximation of $\P^+_\infty$, we can generate a $\delta$-outer approximation of $\P_\infty$, which can be used to compute a finite $(\epsilon,\delta)$-solution to problem \eqref{eq:P}. Note that it is sufficient to consider the set $W$ since this set corresponds, up to the closure, to the cone $\P_\infty^+$ of interest by Proposition~\ref{prop:weightset}. What about the closure? Set $W$ can be computed exactly if it is polyhedral (so closed). Otherwise, it needs to be approximated, in which case an approximation of $W$ is also an approximation of its closure.

From a practical point of view, instead of computing or approximating the cone $W$, we will compute or approximate the bounded convex set
\begin{equation} \label{eq:Wc}
	W_c := W \cap \{ w \in \mathbb{R}^d \mid w^\T c \leq 1\}
\end{equation}
for a fixed $c \in \Int C$ with $\norm{c} = 1$. The next proposition shows that the set $W$ can be approximated {through} an approximation of $W_c$.
	\begin{proposition} \label{prop:Wc}
		Let $W_c$ be as defined in \eqref{eq:Wc} for some $c \in \Int C$ with $\norm{c} = 1$ {and let $\delta \in (0,1)$ be a tolerance}. Assume that $\bar{W}$ is a finite $\delta$-inner approximation of $W_c$ in the sense that it holds 
		 \begin{equation*} \label{eq:projection} 
		 	\bar{W} \subseteq W_c\quad \text{and} \quad d^H (W_c, \conv \bar{W})\leq \delta.	
		 \end{equation*} 
		 Then, $\bar{W}$ is also a finite $\delta$-inner approximation of the cone $W$.
\end{proposition}
\begin{proof}
Consider an element $w \in W \cap B(0,1)$. Note that the Cauchy-Schwarz inequality implies $W \cap B(0,1) \subseteq W_c$, therefore, there exists $\bar{w} \in \conv \bar{W}$ with $\norm{w - \bar{w}} \leq \delta$. Our proof would be finished if $\norm{\bar{w}} \leq 1$. We proceed with the case $\norm{\bar{w}} > 1$ where we show that the orthogonal projection $\frac{w^\T \bar{w}}{\bar{w}^\T \bar{w}} \bar{w}$ provides the desired bound:
Firstly, since $w^\T \bar{w} = \frac{1}{2} \left( \norm{w}^2 + \norm{\bar{w}}^2 - \norm{w - \bar{w}}^2 \right) > {\frac{1}{2}} (1 - \delta^2) > 0$ we know that $\frac{w^\T \bar{w}}{\bar{w}^\T \bar{w}} \bar{w} \in \cone \bar{W}$.
Secondly, $\frac{w^\T \bar{w}}{\bar{w}^\T \bar{w}} \bar{w} \in B(0,1)$ holds since $\norm{ \frac{w^\T \bar{w}}{\bar{w}^\T \bar{w}} \bar{w} } = \frac{\vert w^\T \bar{w}\vert }{\norm{\bar{w}}} \leq \frac{\norm{w} \norm{\bar{w}} }{\norm{\bar{w}}} \leq 1$.
And thirdly, for $\frac{w^\T \bar{w}}{\bar{w}^\T \bar{w}} = \argmin\limits_{\alpha \in \mathbb{R}} \norm{w - \alpha \bar{w}}$ it holds $\norm{ w - \frac{w^\T \bar{w}}{\bar{w}^\T \bar{w}} \bar{w}} \leq \norm{w - \bar{w}} \leq \delta$, which proves the claim.
\end{proof}

{In light of \Cref{rem_1_infty}, let us again address different norms in the context of \Cref{prop:Wc}. Keep in mind that the $\ell_1$ and $\ell_\infty$ norms are relevant for computational purposes.}
\begin{remark}
{
Let $p, {r} \in [1, \infty]$ satisfy $\frac{1}{p} + \frac{1}{{r}} = 1$ and use $c \in \Int C$ with $\norm{c}_{{r}} = 1$ to define the set $W_c$. The following can be shown: \textit{If $\bar{W}$ is a finite $\delta$-inner approximation of $W_c$ in $\ell_p$, then $\bar{W}$ is a finite $\frac{2\delta}{1+\delta}$--inner approximation of the cone $W$ in $\ell_p$.}

We sketch the proof again. Let $w \in W \cap B_p (0,1)$. Since the Cauchy-Schwarz theorem implies $w \in W_c$, there exists $\bar{w} \in \conv \bar{W}$ satisfying $\norm{w - \bar{w}}_p \leq \delta$. Consider the convex optimization problem
\begin{align}
\label{eq_prob_remark2}
\min\limits_{\alpha \geq 0} \norm{w - \alpha \bar{w}}_p.
\end{align}
Since a coercive function attains its minimum over a closed set, there exists an optimal solution $\alpha^*$ satisfying $\norm{w - \alpha^* \bar{w}}_p \leq \norm{w - \bar{w}}_p \leq \delta$ and $\norm{\alpha^* \bar{w}}_p  \leq \norm{w}_p + \norm{\alpha^* \bar{w} - w}_p \leq 1 + \delta$. The claim follows from $\frac{1}{1+\delta} \alpha^* \bar{w} \in \cone \bar{W} \cap B_p (0,1)$ satisfying 
\begin{align*}
\norm{w - \frac{1}{1+\delta} \alpha^* \bar{w}}_p \leq \frac{\delta}{1+\delta} \norm{w}_p +  \frac{1}{1+\delta} \norm{ w- \alpha^* \bar{w}}_p \leq \frac{\delta}{1+\delta} + \frac{\delta}{1+\delta}.
\end{align*}
}

\end{remark}

A cone is determined by its base, such as $W \cap \{ w \in \mathbb{R}^d \mid w^\T c = 1\}$. However, a full-dimensional set $W_c$ is preferable for computational purposes. Alternatively, one could aim to replace the base with a $(q-1)$-dimensional set generating it. Assume without loss of generality that $c_q \neq 0$. Let $c_{-q}\in \R^{q-1}$ denote the first $q-1$ components of $c$ and $w :\R^{q-1} \to \R^q$ be defined as 
$$w (\lambda):= (\lambda^\T, \frac{1-\lambda^\T c_{-q}}{c_q})^\T$$
so that $c^\T w(\lambda) = 1$ holds for all $\lambda \in \R^{q-1}$. Then for the bounded set
\begin{equation}\label{eq:weightset}
	\Lambda:= \{\lambda \in \R^{q-1} \mid w(\lambda)\in C^+, \ \text{\eqref{eq:Dw} is feasible for } w = w(\lambda)\},
\end{equation} 
we have $ W = \cone\{w(\lambda)\in \R^q \mid \lambda \in \Lambda\}$ by construction. 

In particular, for the $q=2$ case, the set $\Lambda$ is a bounded interval and it suffices to solve two scalar problems to compute the bounds 
\begin{equation*}
	\inf\{\lambda \in \R \mid w(\lambda)\in C^+, \ \text{\eqref{eq:Dw} is feasible for } w = w(\lambda)\}
\end{equation*} and 
\begin{equation*}
	\sup\{\lambda \in \R \mid w(\lambda)\in C^+, \ \text{\eqref{eq:Dw} is feasible for } w = w(\lambda)\}.
\end{equation*}

The drawback of considering the $(q-1)$-dimensional set $\Lambda$ arises if the set cannot be computed exactly, but has to be approximated: Approximation error for the set $\Lambda$ is not preserved for the cone $W$ and bound on the tolerance depends on the particular choice of vector $c$. Nevertheless, we consider the approach through the set $\Lambda$ useful at least in two cases: (1) If the set $\Lambda$ can be computed exactly. (2) In the $q=2$ case when the interval $\Lambda$ is approximated through two scalar problems, since solvers for scalar problems can in practice usually handle significantly lower precision than multi-objective problems or algorithms for projection problems.

\section{Computations for special cases} \label{sect:ex}

{The solution approach presented in \Cref{sect:recCone} is applicable for the problems} where Assumption~\ref{assm:const_qual} holds and the set 
\begin{align*}
W =  \{w \in C^+ \mid \text{\eqref{eq:Dw} is feasible}\}
\end{align*}
can be expressed explicitly through the constraints of the dual problem. {In this section, we will discuss three cases for which we can write the dual problem, hence the set $W$, explicitly.}

 We start with a relatively wide class of semidefinite problems{, which is well-studied for the single objective case and has many application areas, see for instance the review paper \cite{BV96}. There are also some studies that consider the class of semidefinite vector optimization problems in the literature, see \cite{EichJahn2008,EichJahn2010,wanka2003multiobjective}. Similar to the single objective case,} the arguments of~\cite{BV96} can be straightforwardly extended to show that linear vector optimization and quadratic convex vector optimization problems with polyhedral ordering cones are special cases of semidefinite vector problems. Nevertheless, we also address linear and quadratic problems individually and provide further observations. 

For the problems we consider below, the set of weights $W$, and consequently also the sets $W_c$ of~\eqref{eq:Wc} and $\Lambda$ of~\eqref{eq:weightset}, have a form of a convex (or polyhedral) projection. Methods for solving convex (or polyhedral) projections can, therefore, be used to approximate (or compute) the set $W_c$ (or the set $\Lambda$). In the light of \Cref{prop:Wc}, we obtain an approximation of~$\mathcal{P}_\infty^+$. Finally, a dual cone of this approximation provides the desired approximation of the recession cone $\mathcal{P}_\infty$ of the upper images as \Cref{prop:aprox_error} shows.

An outer approximation of the recession cone is needed to solve a CVOP in the sense of \Cref{defn:solutionconcept}. The method proposed in this paper can be used to replace the first phase of the algorithm proposed in~\cite{WURKH22}. Keep in mind that if the problem is self-bounded, then the recession cone itself can also be used to solve the problem. If this is not the case, however, an outer approximation of it is needed even if it is possible to compute $\mathcal{P}_\infty$ exactly. In the light of \Cref{prop:weightset}, unless the set $W$ is known to be closed, we need to look for its inner approximation.

\subsection{Semidefinite problems}\label{subsect:semid}
The first class of problems we consider are the semidefinite problems. In the following, $S^k$ denotes the set of symmetric $k \times k$ matrices and $S^k_+$ denotes the set of symmetric, positive semidefinite $k \times k$ matrices.
Consider a semidefinite vector program in inequality form,
\begin{align*}\tag{SDVP} \label{SDVP}
	\text{minimize } & \quad P^\T x \quad \text{ with respect to\ } \leq_C \\ \text{ subject to } & \quad x_1 F_1 + \ldots + x_n F_n + G \preceq 0, 
\end{align*}
for some $P \in \R^{n\times q}, F_1,\ldots,F_n, G \in S^k, k\geq 2$. The weighted sum scalarization for a weight $w\in C^+$ is the scalar semidefinite program
\begin{align*}
\text{minimize } & \quad w^\T P^\T x   \\ \text{ subject to } & \quad x_1 F_1 + \ldots + x_n F_n + G \preceq 0
\end{align*}
and its Lagrange dual is 
\begin{align*}
\text{maximize } & \quad \tr (GZ)   \\ \text{ subject to } & \quad \tr (F_i Z) + e_i^T Pw = 0, \ i\in\{1,\ldots,n\}, \\
& \quad Z \succeq 0.
\end{align*}
We refer a reader interested in the derivation of the dual problem to \cite[Example 5.11]{Boyd}. 

Assumption~\ref{assm:const_qual} on constraint qualification is satisfied if there exists $x\in \R^n$ such that $x_1 F_1 + \ldots + x_n F_n + G \prec 0$, consider \cite[Equation 5.27]{Boyd}. Then the strong duality yields the set $W$ of the convex projection form
$$W = \{w \in C^+ \mid \exists Z \succeq 0 : \ \tr (F_i Z) + e_i^T Pw = 0, \; i = 1,\ldots,n \},$$ {which can be computed by the method presented in \cite{KovacovaRudloff2021}.}

{In the following subsections, we consider two special cases of semidefinite problems for which further simplification and/or observation can be made.}
 
\subsection{Linear problems} \label{subsect:linear}
{The class of linear vector optimization problems is the most studied vector optimization problem class; many solution approaches are available and some of them are already mentioned in \Cref{sect:intro}. Many existing methods are designed to solve bounded problems. Possibly unbounded LVOPs are considered for instance in \cite{Loehne11,psimplex}. It has been shown that the recession cone can be computed via the homogeneous problem, which is again a $q$-dimensional LVOP, see \cite[Section 4.6]{Loehne11}. The parametric simplex method from \cite{psimplex} is a decision space algorithm and provides the recession directions of the upper image at the final stage of the algorithm together with its vertices.} 
	
{In this section, we show that to compute the recession directions of an LVOP, it is sufficient to solve a polyhedral projection problem where the dimension of the problem can be decreased from $q$ to $q-1$. Note that the proposed method is an alternative to computing the recession cone via the homogeneous problem from \cite{Loehne11} as both methods form the first phase of the solution approach considered in this paper.}

Given matrices $P\in \R^{n\times q}, A \in \R^{m\times n}$, a vector $b\in \R^m$, and a polyhedral ordering cone $C$, {we} consider the linear vector optimization problem
\begin{align*}
\tag{LVP} \label{eq:Plin}  
\text{minimize } P^\T x \quad \text{ with respect to\ } \leq_C\quad\text{ subject to } A x \leq b.
\end{align*}
For a weight vector $w \in C^+$, the Lagrange dual \eqref{eq:Dw} of the weighted sum scalarization problem \eqref{eq:Pw} is given by
\begin{align*}
\text{maximize } -b^\T y \quad \text{ subject to } \quad A^\T y = -P w, \quad y\geq 0.
\end{align*}
Applying \Cref{prop:weightset} and \Cref{thm:1}, we obtain
\begin{align}
\label{eq_lin1}
\P_\infty^+ = W = \{w \in C^+ \mid \exists y \geq 0 \ :  -P w = A^\T y\}.
\end{align}
The problem of computing the set~\eqref{eq_lin1} is a polyhedral projection problem. Closure is not needed on the right-hand side of~\eqref{eq_lin1} since the set is a polyhedron. The polyhedral dual cone $\P_\infty^+$ can be computed exactly, rather than approximated, which is appropriate since the linear problem is self-bounded per \Cref{prop:weightset}. 

As we suggested in {\Cref{sect:recCone}}, instead of computing the cone~\eqref{eq_lin1} in $\R^q$, we can compute the $(q-1)$-dimensional set 
	\begin{align*}
		\Lambda = \{\lambda \in \R^{q-1} \mid w(\lambda) \in C^+, \exists y \geq 0 \ :  -P w(\lambda) = A^\T y\},
	\end{align*}
which also corresponds to solving a polyhedral projection problem. Moreover, we know that $\Lambda$ is a closed interval if $q=2$. In this case, {to obtain the bounds of this interval,} it suffices to solve {the following} two scalar linear problems
\begin{align*}
	\text{minimize/maximize } & \quad \lambda \\
	\text{ subject to } & \quad w(\lambda) \in C^+, \\
	& \quad -P w(\lambda)= A^\T y,\\
	& \quad \lambda \in \R, y \geq 0.
\end{align*}

\subsection{Convex quadratic problems}\label{subsect:quad}
{The last special case that we consider is the class of convex quadratic problems, which is a well-established area of mathematical programming in the scalar case. There are also recent papers that consider this class of problems in the multiobjective setting in different contexts, see \cite{chuong2022robust,eichfelder2022note,jayasekara2023solving,oberdieck2016multi}. In this section, we show that if the convex quadratic vector optimization problem}
contains at least one quadratic constraint, then the problem is bounded.
Moreover, below we identify several conditions under which it holds $\mathcal{P}_\infty^+ = C^+$.

We consider the following convex quadratic vector optimization problem
\begin{align*}\tag{QVP} \label{QVP}
\text{minimize } & \quad f(x) \quad \text{ with respect to\ } \leq_C \\ \text{ subject to } & \quad x^\T Q_j x + c_j^\T x + r_j \leq 0, \quad j\in \{1,\ldots,p \}, \\
& \quad A x \leq b, 
\end{align*}
where $Q_j \in S^{n}_+ \setminus \{0\}, c_j \in \R^n, r_j \in \R$ for $j\in\{1,\ldots,p\}$, $A \in \R^{m\times n}, b\in \R^m$, and the $C$-convex objective function $f = (f_1, \ldots,f_q)^\T:\R^n \to \R^q$ is given by $f_i(x) = x^\T P_i x + d_i^\T x$ with $P_i\in S^{n}, d_i \in \R^n$ for $i = 1,\ldots,q.$ Note that $f$ is $C$-convex if and only if for all $w\in C^+$ is $w^\T f$ convex, or equivalently $\sum_{i=1}^q w_i P_i \succeq 0$. In particular, for $C \supseteq \R^q_+$ a convexity of each objective $f_1, \dots, f_q$ implies $C$-convexity of $f$. For $C=\R^q_+$, the converse also holds. 

Now let us look at what we can learn about the problem.
The weighted sum scalarization for a weight vector $w \in C^+$
\begin{align*}
\text{minimize } & x^\T \left(\sum_{i=1}^q w_i P_i\right) x + \left(\sum_{i=1}^q w_i d_i \right)^\T x \\ 
\text{ subject to } & \quad x^\T Q_j x + c_j^\T x + r_j \leq 0, \quad j\in \{1,\ldots,p \}, \\
& \quad A x \leq b 
\end{align*}
yields a dual function
\begin{align*}
g(\nu,\mu) =  &\inf_{x \in \R^n} \left(x^\T \left(\sum_{i=1}^q w_i P_i + \sum_{j=1}^p \nu_j Q_j \right) x + \left(\sum_{i=1}^q w_i d_i + \sum_{j=1}^p \nu_j c_j + A^\T \mu \right)^\T x\right) \\ &+ \nu^\T r -\mu^\T b. 
\end{align*}
Keeping \Cref{thm:1} in mind, we are interested in the weights $w$ for which the dual problem is feasible.
Given the infimum term in the dual function, we have feasibility in two cases: if the quadratic expression in $x$ is convex, or if the quadratic expression in $x$ is constant. This yields the following form of set $W$,
\begin{align} \label{eq:PinftyQCQP}
&W = \left\lbrace w \in C^+ \mid \exists \nu \in \R^p_+ : 0 \neq \sum_{i=1}^q w_i P_i + \sum_{j=1}^p \nu_j Q_j \succeq 0 \right\rbrace \  \cup \\ 
& \left\lbrace w \in C^+ \mid \exists \nu \in \R^p_+, \mu \in \R^m_+ : \ \sum_{i=1}^q w_i P_i + \sum_{j=1}^p \nu_j Q_j = 0,  \sum_{i=1}^q w_i d_i + \sum_{j=1}^p \nu_j c_j + A^\T \mu = 0\right\rbrace . \notag
\end{align}

Using the structure of $W$ given by \eqref{eq:PinftyQCQP}, we show in the following two propositions that either the set $W$ itself or its closure is equal to $C^+$ in some standard cases.
\begin{proposition} \label{prop:quad1}
 	Consider problem \eqref{QVP}. In each of the following cases, $W = \P_\infty^+ = C^+$ holds, in particular, the problem is bounded. 
 	\begin{enumerate}[(a)]
 		\item There is at least one nonlinear constraint, that is, $p>0$.
 		\item $P_1,\ldots,P_q \in S^n$ are linearly independent. 
 	\end{enumerate}
 	\end{proposition}
 \begin{proof} For each case, we will show $W=C^+$. This implies $W = \P_\infty^+$ and by \Cref{prop:weightset}, the problem is self-bounded. Indeed, it is bounded as we also have $\P_\infty^+ = C^+$.
 	\begin{enumerate}[(a)]
 	\item By convexity, we have $Q_1, \dots, Q_p \succeq 0$ and $\sum_{i=1}^q w_i P_i \succeq 0$ for arbitrary $w \in C^+$. 
 	If $\sum_{i=1}^q w_i P_i \neq 0$, then the choice of $\nu = 0$ gives $0 \neq \sum_{i=1}^q w_i P_i + \sum_{j=1}^p \nu_j Q_j \succeq 0$. If $\sum_{i=1}^q w_i P_i = 0$, then the choice of $\nu_1 = 1, \nu_2 = \dots, \nu_p = 0$ gives $0 \neq \sum_{i=1}^q w_i P_i + \sum_{j=1}^p \nu_j Q_j \succeq 0$. This shows that 
 	$$\left\lbrace w \in C^+ \mid \exists \nu \in \R^p_+ : 0 \neq \sum_{i=1}^q w_i P_i + \sum_{j=1}^p \nu_j Q_j \succeq 0 \right\rbrace = C^+.$$
 	Together with \eqref{eq:PinftyQCQP}, this implies that $W=C^+$.
 	\item By (a), it is sufficient to consider problems without nonlinear constraints, that is, $p=0$. In this case, the cone $W$ given by \eqref{eq:PinftyQCQP} simplifies to 
 	\begin{align*}
 		\left\lbrace w \in C^+ \mid 0 \neq \sum_{i=1}^q w_i P_i \succeq 0 \right\rbrace  \cup 
 		\left\lbrace w \in C^+ \mid \sum_{i=1}^q w_i P_i =0, \exists \mu \geq 0 : \  \sum_{i=1}^q w_i d_i + A^\T \mu = 0\right\rbrace.
 	\end{align*} Since the $C$-convexity of the objective implies $\sum_{i=1}^q w_i P_i \succeq 0$ for all $w \in C^+$, we can write $W = (C^+\setminus W_1) \cup W_2,$ where 
 	\begin{align} \label{eq:W1set}
 		\begin{split}
 		W_1 &:= \left\lbrace w \in C^+ \mid \sum_{i=1}^q w_i P_i = 0 \right\rbrace , \\
 		W_2 &:= \left\lbrace w \in C^+ \mid \exists \mu \geq 0 : \ \sum_{i=1}^q w_i P_i =0,   \sum_{i=1}^q w_i d_i + A^\T \mu = 0\right\rbrace. 
 		\end{split}
 	\end{align} 
 	\end{enumerate}
 	If the matrices $P_1, \dots, P_q$ are linearly independent, then $W_1 = \{0\}$, since  $\sum_{i=1}^q w_i P_i = 0$	occurs only for $w=0$. Since $0 \in W_2$ and $W = (C^+\setminus W_1) \cup W_2$, we conclude $W=C^+$.
\end{proof}

\begin{proposition}\label{prop:quad2}
	Consider problem \eqref{QVP} and assume that the problem is nonlinear, that is, there is at least one nonlinear constraint or objective function. If $C = \R^q_+$, then $\P_\infty^+ = \R^q_+$. 
\end{proposition} 

\begin{proof}

By \Cref{prop:quad1} (a), it is sufficient to consider problems without nonlinear constraints, that is, $p=0$. In this case, $W = (C^+\setminus W_1) \cup W_2,$ where $W_1,W_2$ are as in \eqref{eq:W1set}. If $W_1 = \{0\}$, then $\P_\infty^+ = C^+$ follows since $\P_\infty^+ = \cl W$. Assume $w \in W_1 \setminus\{0\}$. Noting that $C^+ = \R^q_+$, $w_j>0$ for some $j\in \{1,\ldots,q\}$. Consider the diagonal elements of the matrix $\sum_{i=1}^q w_i P_i = 0$. Since the matrices $P_1, P_2, \dots, P_q$ are positive semidefinite for $C=\R^q_+$, all of their diagonal elements are nonnegative. Then, $\sum_{i=1}^q w_i P_i = 0$ implies for $w_j > 0$ that all the diagonal elements of matrix $P_j$ are zero and, therefore, $P_j$ is the zero matrix. 

Since the problem is not linear, there exists $i \in\{1, \dots, q\}$ with $P_i \neq 0$. Then, for any $w \in W_1$ we can construct a sequence of $w^{(n)} := w + \frac{1}{n}e^i \in \mathbb{R}^q_+ \setminus W_1$ converging to $w$. Hence, we conclude $\P_\infty^+ = \cl \left( \mathbb{R}^q_+ \setminus W_1 \right) = \mathbb{R}^q_+$.
\end{proof}

We see that the computation of $\P_\infty^+$ is only relevant if $C \neq \R^q_+$, \eqref{QVP} has only linear constraints and $P_1, \dots, P_q$ are linearly dependent. In that case, it can be done via computing sets $W_1, W_2$ given by \eqref{eq:W1set} and setting $\P_\infty^+ = \cl\big((C^+\setminus W_1) \cup W_2\big).$ As long as the ordering cone is polyhedral, $W_2$ is in the form of a polyhedral projection, so $\P_\infty^+$ can be obtained through computations with polyhedra.

\section{Numerical Examples} \label{sect:comp}
In this section, we provide numerical examples to illustrate the proposed solution methodology. We consider a two-dimensional {and a three-dimensional} linear problem and two semidefinite programming problems with different objective functions minimized over the same feasible set.
\begin{example}\label{example1} Consider the illustrative two-dimensional linear example
\begin{align*}
\min \begin{pmatrix} x_1 \\ x_2 \end{pmatrix} \text{ w.r.t. } \leq_{\R^2_+} \text{ s.t. } 
\begin{pmatrix}
-4 & -1 \\ -2 & -1 \\ -1 & -1 \\ -1 & -2 \\ -1 & -4
\end{pmatrix}
\begin{pmatrix} x_1 \\ x_2 \end{pmatrix} \leq
\begin{pmatrix}
-5 \\-5 \\ -4 \\ -5 \\ -5
\end{pmatrix}.
\end{align*}
As outlined in Section~\ref{subsect:linear}, to identify the recession cone of the upper image, it suffices to solve two scalar linear problems, 
\begin{align*}
\text{minimize / maximize } \; \lambda \quad \text{ subject to } \begin{pmatrix} \lambda\\ 1- \lambda \end{pmatrix} \geq 0, \;  y \geq 0, \;
\begin{pmatrix}
4 & 2 & 1 & 1 & 1 \\
1 & 1 & 1 & 2 & 4
\end{pmatrix} y = \begin{pmatrix} \lambda\\ 1- \lambda \end{pmatrix}.
\end{align*}
 These yield the optimal values $\lambda_{\min} = 0.2$ and $\lambda_{\max} = 0.8$, which generate the dual cone $W = \P_{\infty}^+ = \cone \left\lbrace \begin{pmatrix} 0.2 \\ 0.8 \end{pmatrix}, \begin{pmatrix} 0.8 \\ 0.2 \end{pmatrix} \right\rbrace$ and, consequently, the recession cone of upper image $\P_{\infty} = \cone \left\lbrace \begin{pmatrix} -1 \\ 4 \end{pmatrix}, \begin{pmatrix} 4 \\ -1 \end{pmatrix} \right\rbrace$. The dual cone of weights $W$, the recession cone $\mathcal{P}_\infty$, and the upper image $\mathcal{P}$ are depicted in Figure~\ref{fig1}.

\begin{figure}
\centering
\begin{subfigure}{.4\textwidth}
  \centering
  \includegraphics[width=\linewidth]{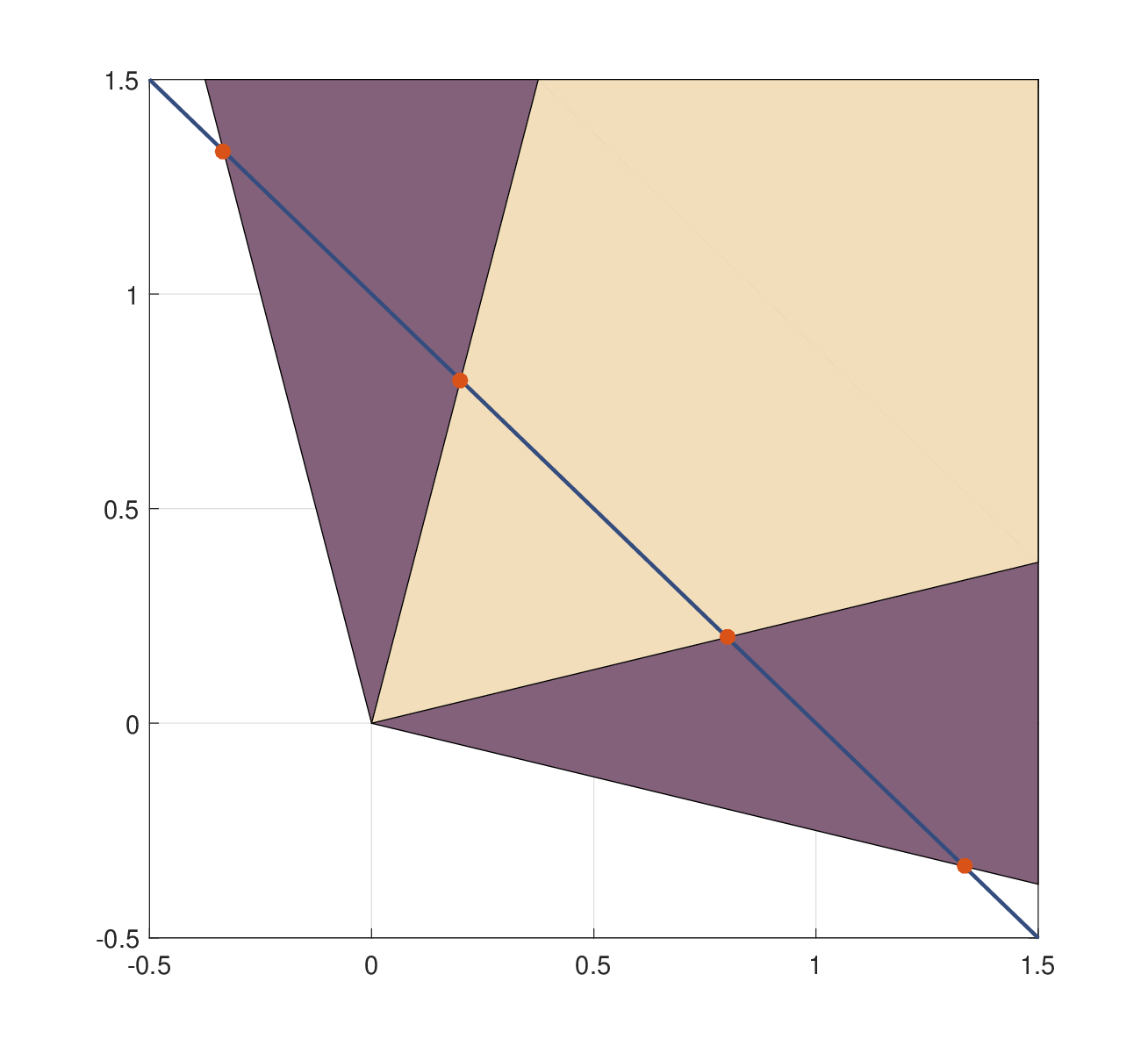}
\end{subfigure}%
\begin{subfigure}{.4\textwidth}
  \centering
  \includegraphics[width=\linewidth]{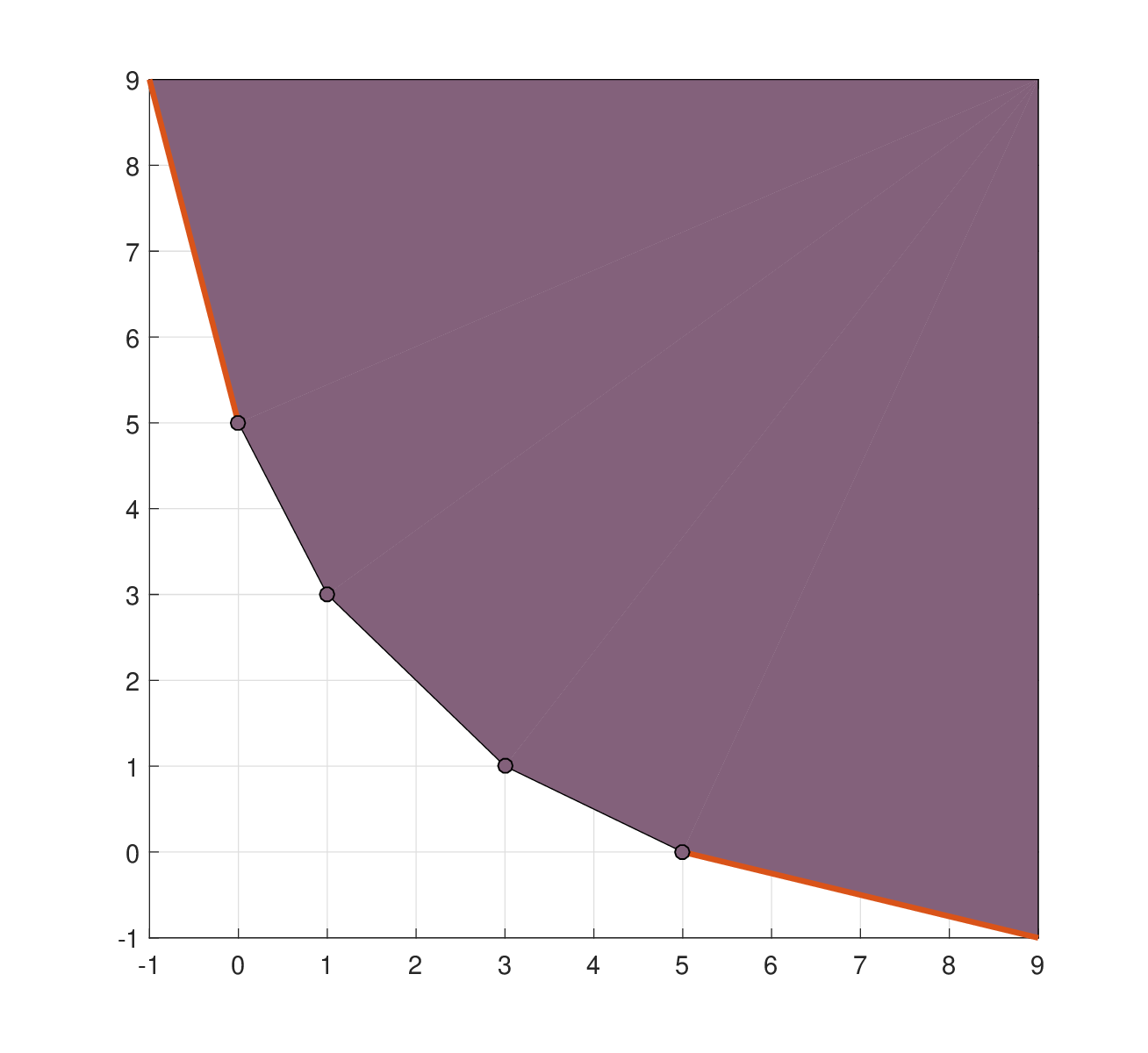}
\end{subfigure}
\caption{\label{fig1}Linear problem from Example~\ref{example1}. Left: Recession cone $\mathcal{P}_\infty$ (dark purple) and the set of weights $W$ (lighter yellow). The depicted line $w_1 + w_2 = 1$ represents the choice of base of the cones, which is used for the two scalar problems solved. Right: Upper image with highlighted recession direction. }
\end{figure}

\end{example}

{\begin{example}\label{ex:lin3d}
	Consider the three-dimensional linear problem 
	\begin{align*}
		\min \begin{pmatrix} x_1 \\ x_2 \\x_3 \end{pmatrix} \text{ w.r.t. } \leq_{C} \text{ s.t. } 
		\begin{pmatrix}
			-1 & -1 & -1 \\ -4 & -1 & -1 \\ -1 & -4 & -1 \\ -1 & -1 & -4 \\ -1 & -1 & \:\:\: 0 \\ -1 & \:\:\: 0 & -1 \\ \:\:\:0 & -1 & -1 
		\end{pmatrix}
		\begin{pmatrix} x_1 \\ x_2 \\ x_3 \end{pmatrix} \leq
		\begin{pmatrix}
			-16 \\-16 \\ -16 \\ -16 \\ -10 \\ -10 \\ -10
		\end{pmatrix},
	\end{align*}
where the ordering cone is $C = \cone \left\lbrace \begin{pmatrix} 4 \\ 2 \\ 2 \end{pmatrix}, \begin{pmatrix} 2 \\ 4 \\ 2 \end{pmatrix}, \begin{pmatrix} 4 \\ 0 \\ 2 \end{pmatrix}, \begin{pmatrix} 1 \\ 0 \\ 2 \end{pmatrix}, \begin{pmatrix} 0 \\ 1 \\ 2 \end{pmatrix}, \begin{pmatrix} 0 \\ 4 \\ 2 \end{pmatrix} \right\rbrace$. We take $c = (1, 1,1)^\T \in \Int C$, that is, $w: \R^2 \to \R^3$ is given by $w(\lambda)=(\lambda_1,\lambda_2,1-\lambda_1-\lambda_2)$. As explained in \Cref{subsect:linear}, it is possible to compute the set $\Lambda \subseteq \R^2$ by solving a bounded linear projection problem. The sets $W, \P_{\infty},$ and $\Lambda$ are displayed in \Cref{fig3d}.

\begin{figure}
	\centering
	\begin{subfigure}{.55\textwidth}
		\centering
		\includegraphics[width=\linewidth]{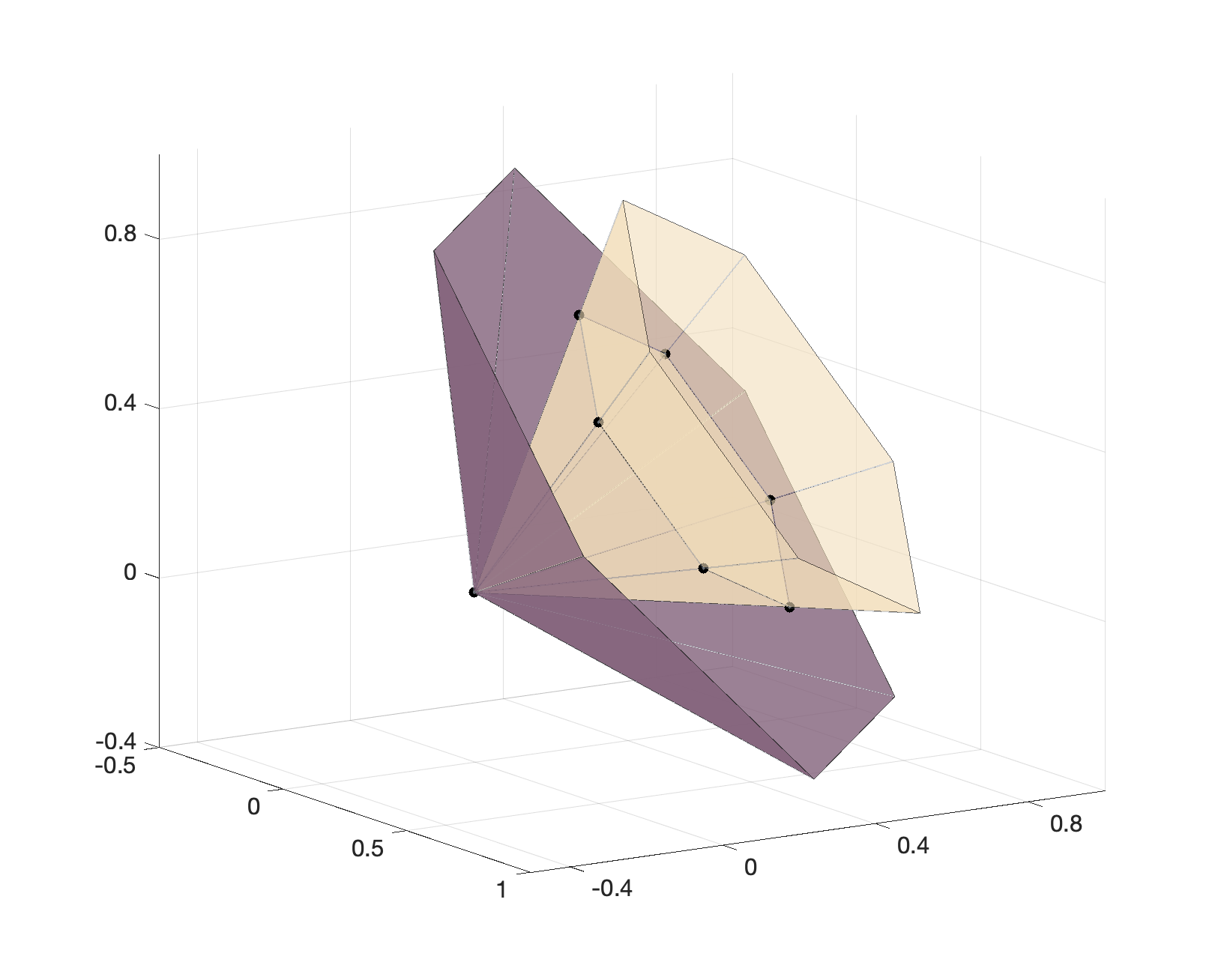}
	\end{subfigure}%
	\begin{subfigure}{.25\textwidth}
		\centering
		\includegraphics[width=\linewidth]{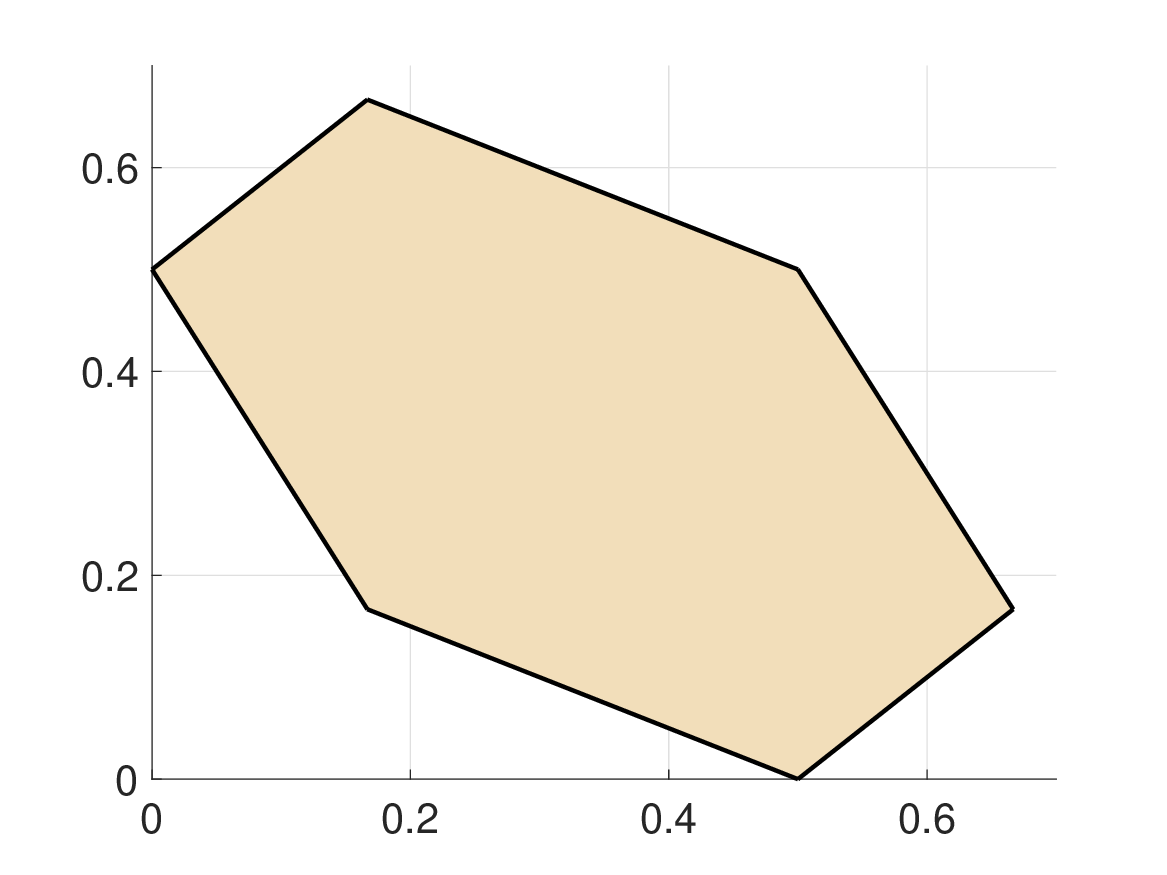}
	\end{subfigure}
	\caption{\label{fig3d} {Linear problem from Example~\ref{ex:lin3d}. Left: Recession cone $\mathcal{P}_\infty$ (dark purple) and the set of weights $W$ (lighter yellow). The depicted plane $w_1 + w_2 + w_3 = 1$ represents the choice of the base of cone $W$. Right: The set $\Lambda$.} }
\end{figure}

	\end{example}
}
\begin{example}\label{example2}
We consider the semidefinite problem
\begin{align*}
\text{minimize } & \quad P^\T x  \quad \text{ with respect to\ } \leq_{\mathbb{R}^3_+} \\ \text{ subject to } & \quad x_1 \begin{pmatrix} -1 & 2 \\ 2 & 4 \end{pmatrix} + x_2 \begin{pmatrix} 2 & 1 \\ 1 & -1 \end{pmatrix} + x_3 \begin{pmatrix} 2 & 2 \\ 2 & 2 \end{pmatrix} \preceq 0
\end{align*}
for objectives given by matrices
\begin{align*}
P_1 = \begin{pmatrix}
0 & 0 & -1\\ -1 & 1 & 0\\ 1 & 1 & -1
\end{pmatrix} \text{ and } P_2 = \begin{pmatrix}
1 &0 &-1\\ -1 &1 &0\\ 0 &0 &-1
\end{pmatrix}.
\end{align*}

We find an approximation of the cone of recession directions
$$W = \{w \in \mathbb{R}^3_+ \mid \exists Z \succeq 0: \ \tr (F_i Z) + e_i^T Pw = 0, \; i = 1,2,3 \}$$
through the convex projection problem of (approximately) computing the set
\begin{align*}
W_c = \{ w \in \mathbb{R}^3_+ \mid \exists Z \succeq 0: \ \tr (F_i Z) + e_i^T Pw = 0, \; i = 1,2,3, \; w_1 + w_2 + w_3 \leq \sqrt{3} \}.
\end{align*}
The convex projection yields both inner and outer approximations of the set $W_c$, which generate inner and outer approximations of cones of $W$ and $\mathcal{P}_\infty$. 
All of them are displayed in Figure~\ref{fig2} for the problem with objective $P_1$. Outer approximation of $\mathcal{P}_\infty$, for which approximation tolerance is guaranteed by Propositions~\ref{prop:aprox_error} and~\ref{prop:Wc}, is needed as a part of a solution.


In Figure~\ref{fig4} we use the problem with objective $P_2$ to compare approximations obtained via the set $W_c$ and the set $\Lambda$. Recall that we only have tolerance guarantees for approach through the set $W_c$.

\begin{figure}[H]
\centering
\begin{subfigure}{.4\textwidth}
	\centering
	\includegraphics[width=\linewidth]{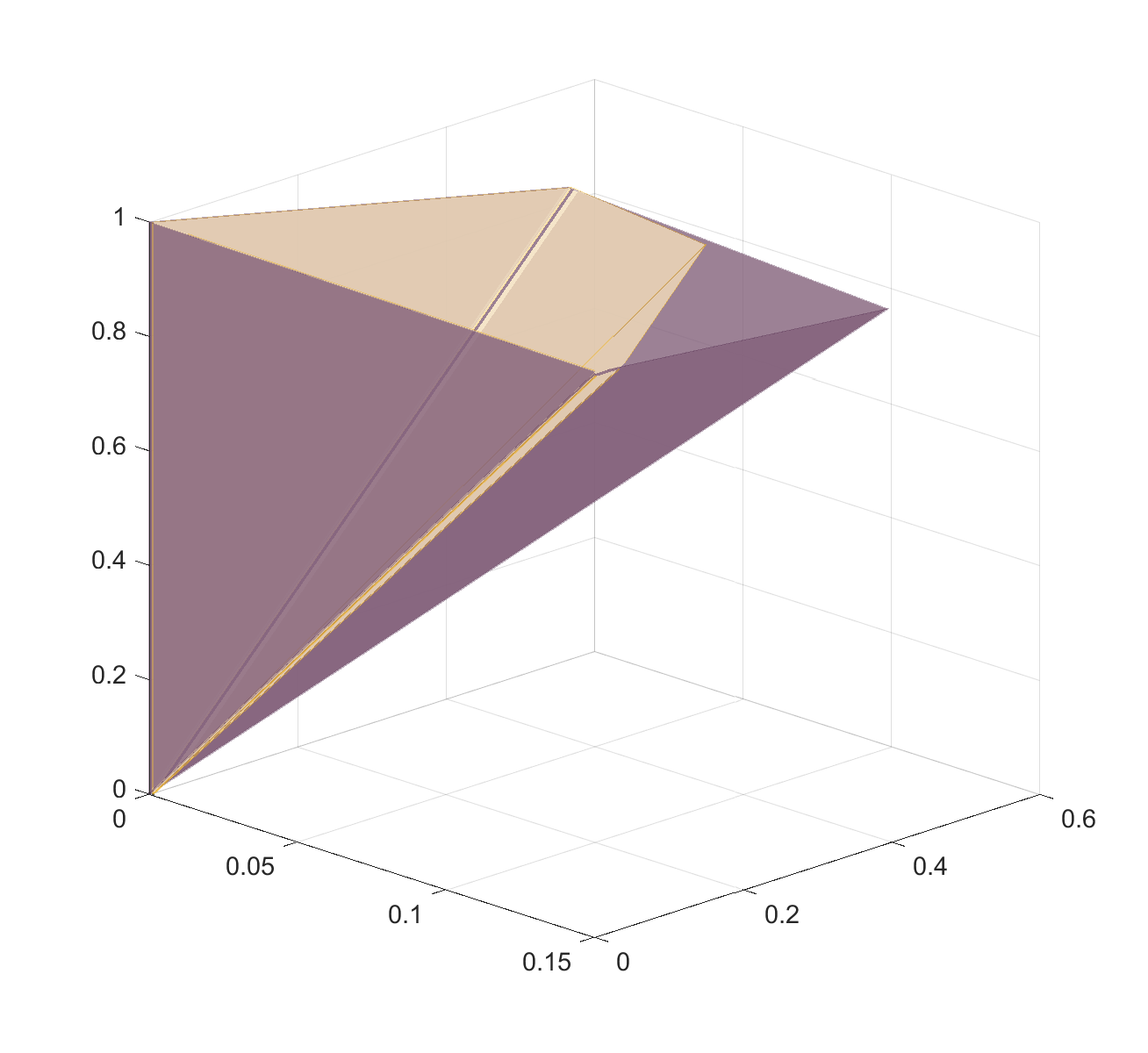}
\end{subfigure}%
\begin{subfigure}{.4\textwidth}
	\centering
	\includegraphics[width=\linewidth]{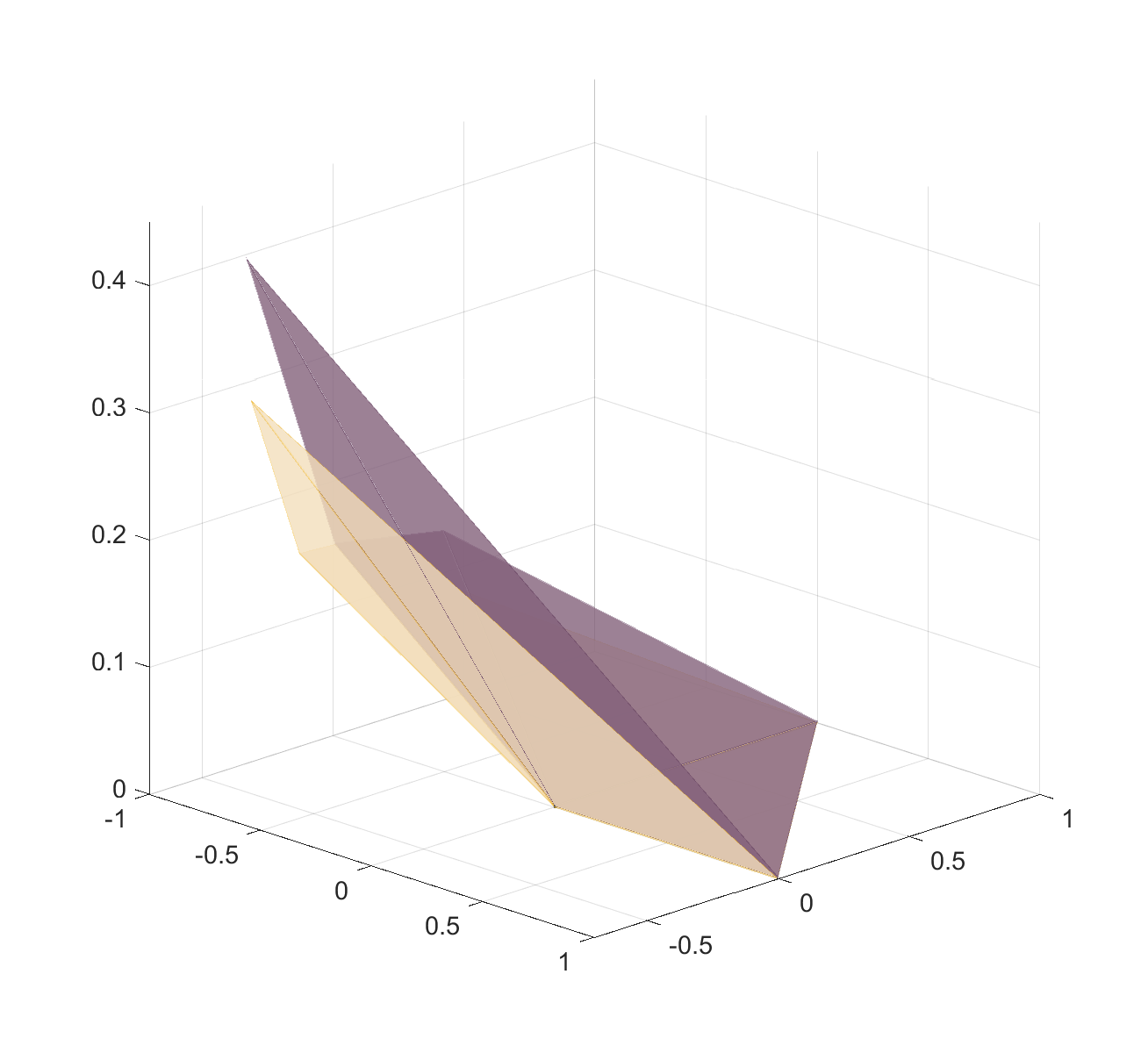}
\end{subfigure}\\
\begin{subfigure}{.4\textwidth}
	\centering
	\includegraphics[width=\linewidth]{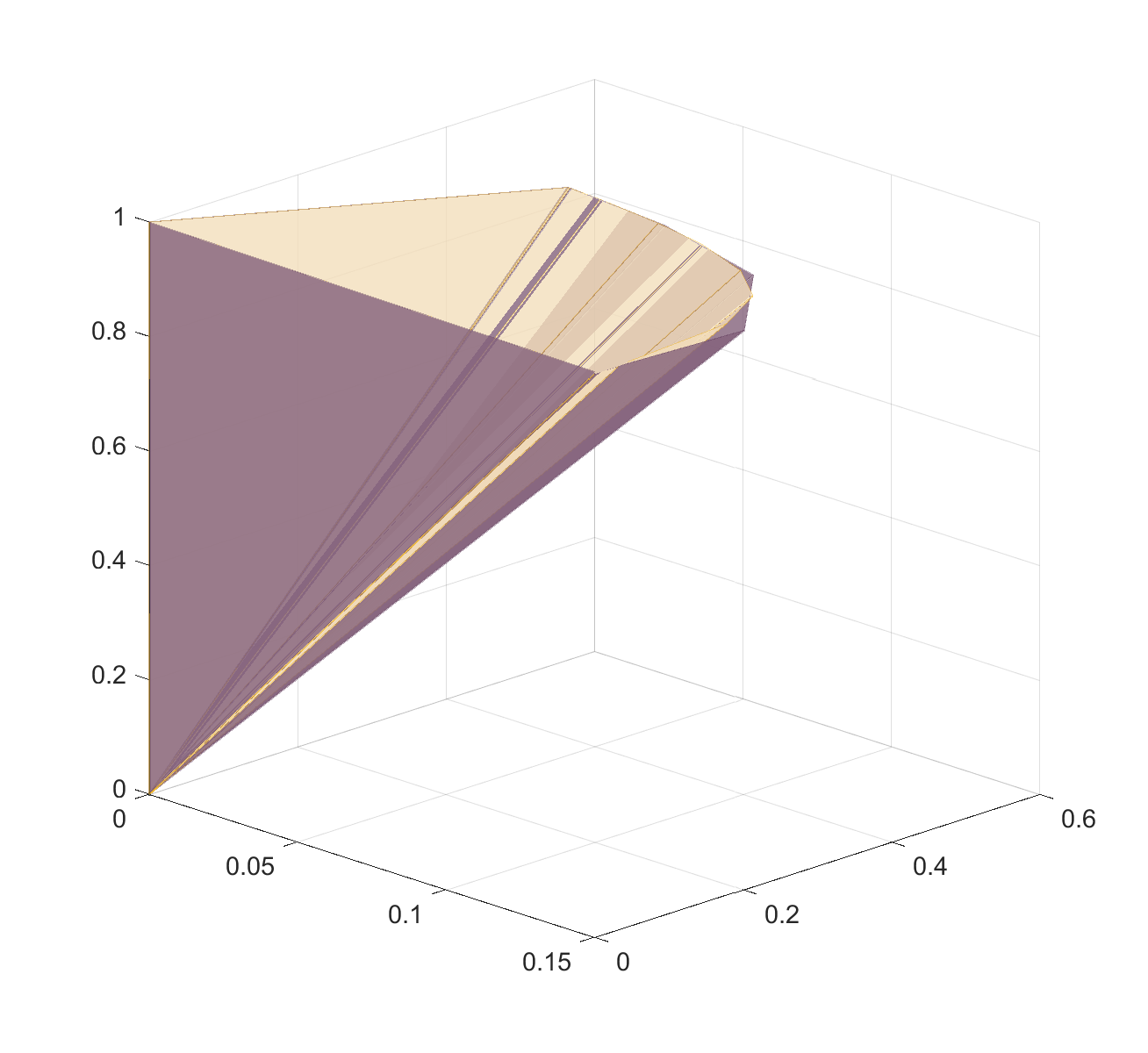}
\end{subfigure}%
\begin{subfigure}{.4\textwidth}
	\centering
	\includegraphics[width=\linewidth]{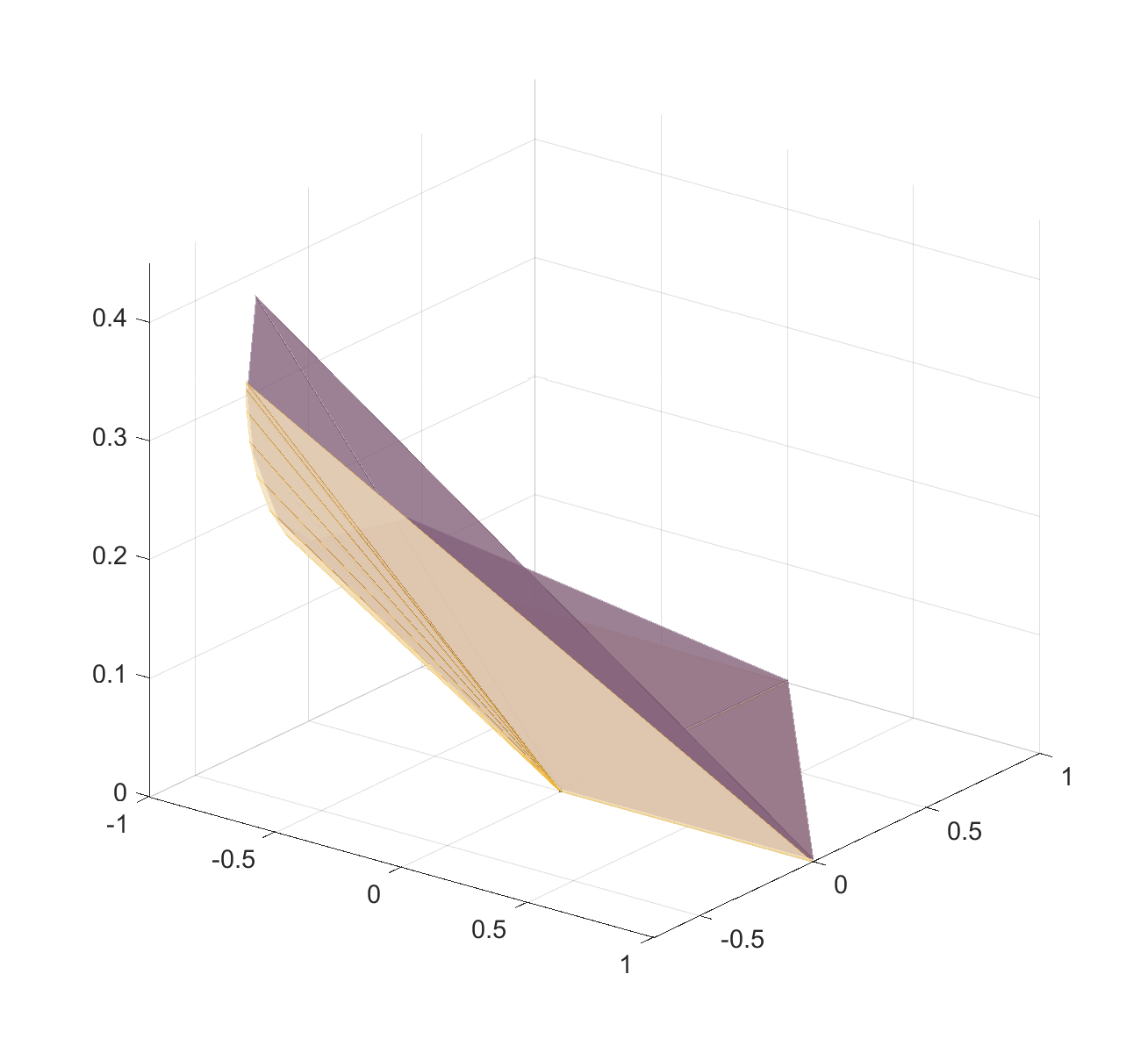}
\end{subfigure} 
\caption{\label{fig2} Semidefinite problem from Example~\ref{example2} with objective $P_1$ solved for $\epsilon = 0.08$ (top) and $\epsilon=0.01$ (bottom). Displayed are inner and outer approximations of the set $W$ (left) and the recession cone $\mathcal{P}_\infty$ (right).}
\end{figure}


\begin{figure}
\centering
\begin{subfigure}{.4\textwidth}
  \centering
  \includegraphics[width=\linewidth]{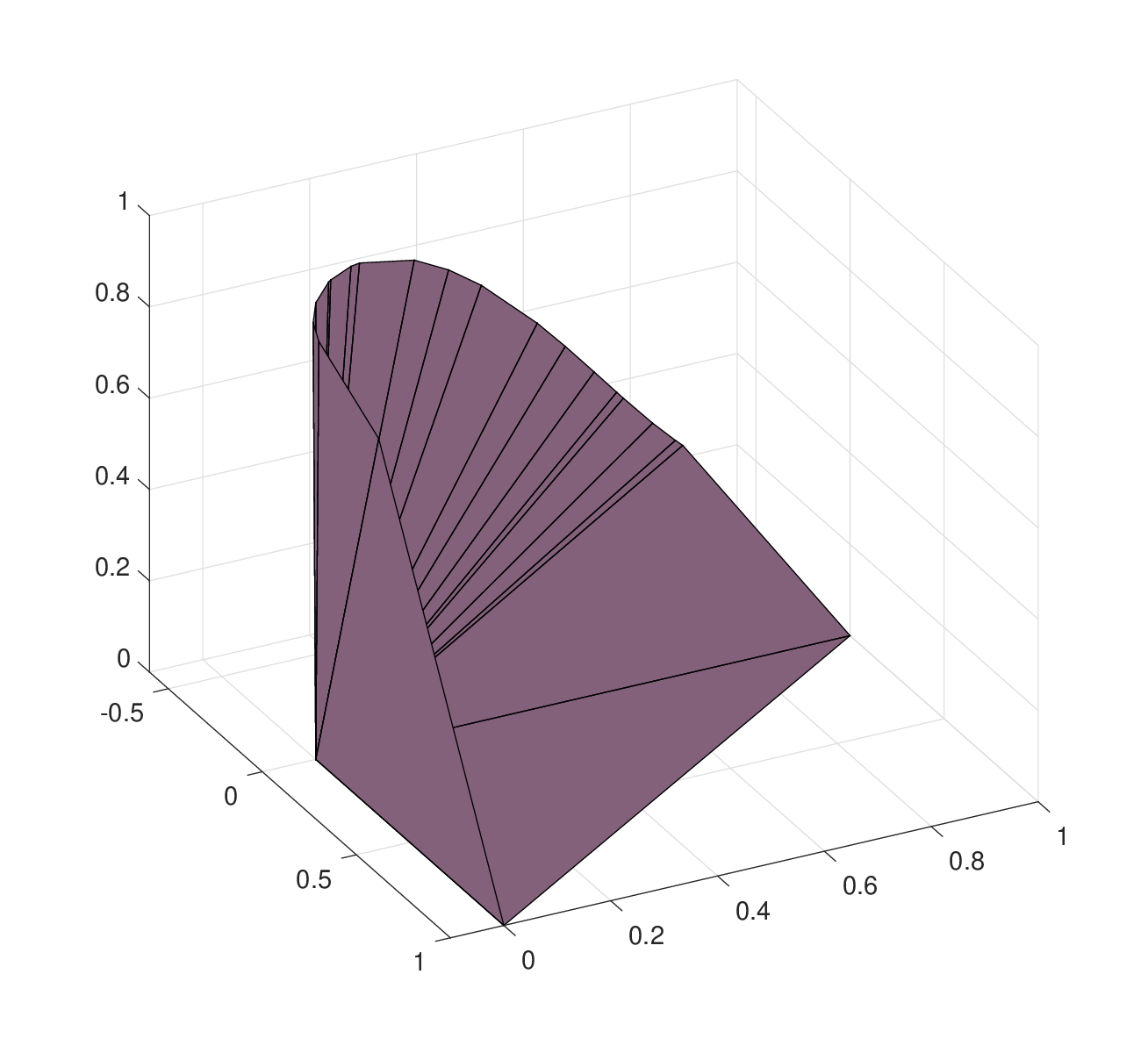}
\end{subfigure}%
\begin{subfigure}{.4\textwidth}
  \centering
  \includegraphics[width=\linewidth]{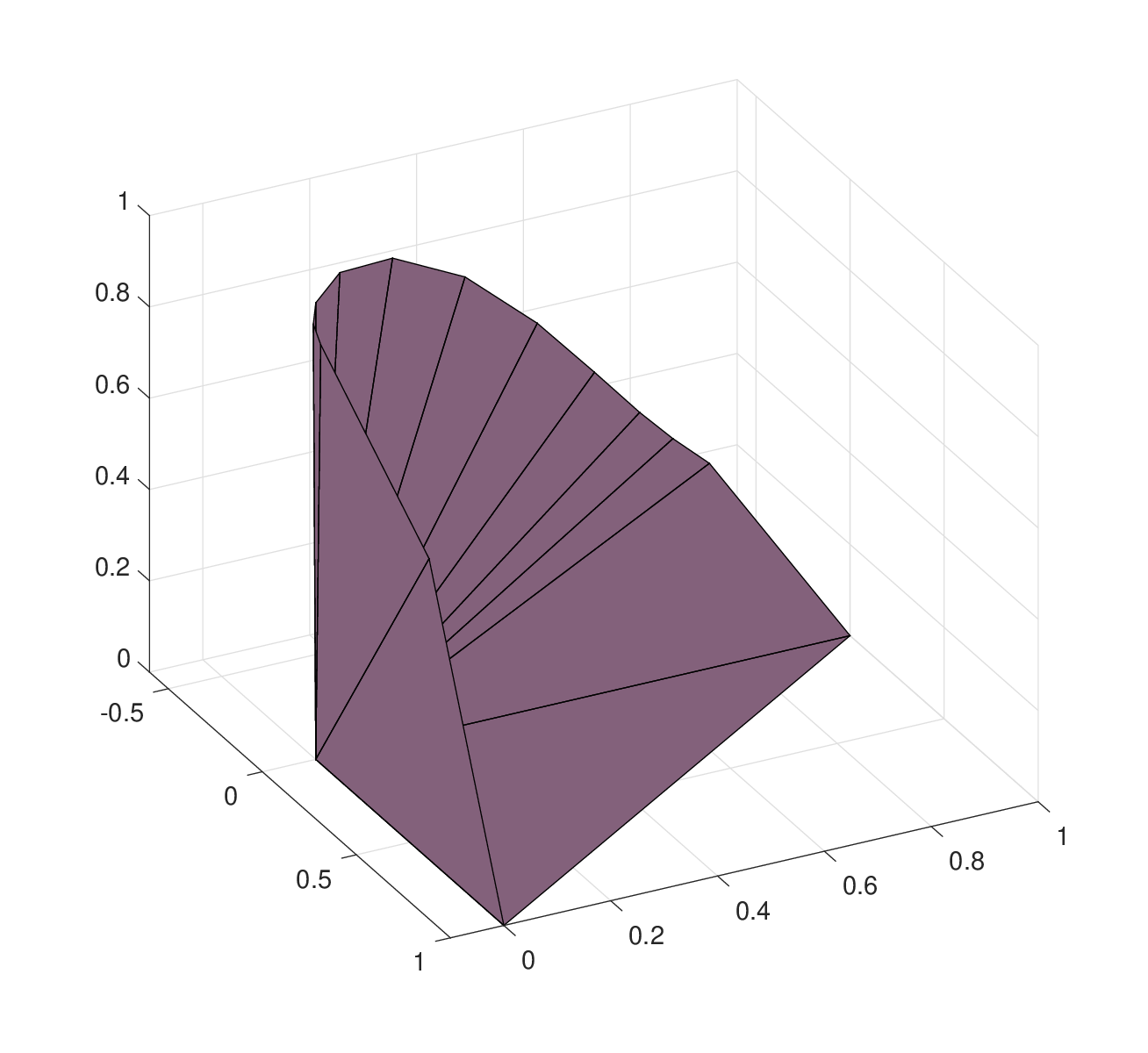}
\end{subfigure}
\caption{\label{fig4}Recession cone of the semidefinite problem from Example~\ref{example2} with objective $P_2$. Compare the approximations of $\mathcal{P}_\infty$ obtained via the set $W_c$ (left) and via the set $\Lambda$ (right), both convex projections were solved for tolerance $\epsilon = 0.005$. 
}
\end{figure}

\end{example}

\section*{Acknowledgments}
G. Kov\'{a}\v{c}ov\'{a} and F. Ulus acknowledge support from the OeNB anniversary fund, project number 17793. {The authors thank Daniel D\"orfler for bringing their attention to the results of \cite{walkup1967continuity}. }
\section*{Declarations}
This manuscript has no associated data. 

\bibliographystyle{plain}
\bibliography{references2}

\end{document}